\documentclass[12pt,oneside]{amsart}

\input{imports_and_macros}
\usepackage{amsrefs}
\usepackage{geometry}                % See geometry.pdf to learn the layout options. There are lots.
\usepackage{graphicx}
\usepackage{amssymb}
\usepackage{epstopdf}
\usepackage{color}
\usepackage{bbm, dsfont,hyperref}
\usepackage{tikz}
\usetikzlibrary{3d}

\usetikzlibrary{patterns.meta} 

\numberwithin{equation}{section}

\usepackage{verbatim}

\usepackage{amssymb,amsthm,amsmath,amssymb,wrapfig,dsfont}
\DeclareFontFamily{OML}{rsfs}{\skewchar\font'177}
\DeclareFontShape{OML}{rsfs}{m}{n}{ <5> <6> rsfs5 <7> <8> <9>
	rsfs7 <10> <10.95> <12> <14.4> <17.28> <20.74> <24.88> rsfs10 }{}
\DeclareMathAlphabet{\mathfs}{OML}{rsfs}{m}{n}

\usepackage{oplotsymbl}

\usepackage{ulem}
\theoremstyle{openproblem}

\usetikzlibrary{calc}
\usepackage{amsmath}
\usepackage{mathrsfs}            %数学花体 

\newcommand{\Pois}{\mathrm{Pois}}

\newcommand{\RNum}[1]{\uppercase\expandafter{\romannumeral #1\relax}}

\newcommand{\specificthanks}[1]{\@fnsymbol{#1}}

\usetikzlibrary{shapes}

\DeclareFontFamily{U}{mathx}{}
\DeclareFontShape{U}{mathx}{m}{n}{<-> mathx10}{}
\DeclareSymbolFont{mathx}{U}{mathx}{m}{n}
\DeclareMathAccent{\widehat}{0}{mathx}{"70}
\DeclareMathAccent{\widecheck}{0}{mathx}{"71}

\usetikzlibrary{patterns}

\def\hexagonsize{0.25cm}
\pgfdeclarepatternformonly
{hexagons}% name
{\pgfpointorigin}% lower left
{\pgfpoint{3*\hexagonsize}{0.866025*2*\hexagonsize}}%  upper right
{\pgfpoint{3*\hexagonsize}{0.866025*2*\hexagonsize}}%  tile size
{% shape description
	\pgfsetlinewidth{0.4pt}
	\pgftransformshift{\pgfpoint{0mm}{0.866025*\hexagonsize}}
	\pgfpathmoveto{\pgfpoint{0mm}{0mm}}
	\pgfpathlineto{\pgfpoint{0.5*\hexagonsize}{0mm}}
	\pgfpathlineto{\pgfpoint{\hexagonsize}{-0.866025*\hexagonsize}}
	\pgfpathlineto{\pgfpoint{2*\hexagonsize}{-0.866025*\hexagonsize}}
	\pgfpathlineto{\pgfpoint{2.5*\hexagonsize}{0mm}}
	\pgfpathlineto{\pgfpoint{3*\hexagonsize+0.2mm}{0mm}}
	\pgfpathmoveto{\pgfpoint{0.5*\hexagonsize}{0mm}}
	\pgfpathlineto{\pgfpoint{\hexagonsize}{0.866025*\hexagonsize}}
	\pgfpathlineto{\pgfpoint{2*\hexagonsize}{0.866025*\hexagonsize}}
	\pgfpathlineto{\pgfpoint{2.5*\hexagonsize}{0mm}}
	\pgfusepath{stroke}
}

\def\squaresize{0.3cm}
\pgfdeclarepatternformonly
{squares}% name
{\pgfpointorigin}% lower left
{\pgfpoint{\squaresize}{\squaresize}}%  upper right
{\pgfpoint{\squaresize}{\squaresize}}%  tile size
{% shape description
	\pgfsetlinewidth{0.4pt}
	\pgfpathmoveto{\pgfpointorigin}
	\pgfpathlineto{\pgfpoint{\squaresize}{0cm}}
	\pgfpathlineto{\pgfpoint{\squaresize}{\squaresize}}
	\pgfpathlineto{\pgfpoint{0cm}{\squaresize}}
	\pgfpathclose
	\pgfusepath{stroke}
}

\def\hexagonsize{0.25cm}
\pgfdeclarepatternformonly
{rotatedHexagons}% new name for the rotated shape
{\pgfpointorigin}% lower left
{\pgfpoint{3*\hexagonsize}{0.866025*2*\hexagonsize}}%  upper right
{\pgfpoint{3*\hexagonsize}{0.866025*2*\hexagonsize}}%  tile size
{% shape description
	\pgfsetlinewidth{0.4pt}
	\pgftransformshift{\pgfpoint{0mm}{0.866025*\hexagonsize}}
	\pgftransformrotate{90}  % rotation
	\pgfpathmoveto{\pgfpoint{0mm}{0mm}}
	\pgfpathlineto{\pgfpoint{0.5*\hexagonsize}{0mm}}
	\pgfpathlineto{\pgfpoint{\hexagonsize}{-0.866025*\hexagonsize}}
	\pgfpathlineto{\pgfpoint{2*\hexagonsize}{-0.866025*\hexagonsize}}
	\pgfpathlineto{\pgfpoint{2.5*\hexagonsize}{0mm}}
	\pgfpathlineto{\pgfpoint{3*\hexagonsize+0.2mm}{0mm}}
	\pgfpathmoveto{\pgfpoint{0.5*\hexagonsize}{0mm}}
	\pgfpathlineto{\pgfpoint{\hexagonsize}{0.866025*\hexagonsize}}
	\pgfpathlineto{\pgfpoint{2*\hexagonsize}{0.866025*\hexagonsize}}
	\pgfpathlineto{\pgfpoint{2.5*\hexagonsize}{0mm}}
	\pgfusepath{stroke}
}

\usepackage{amsfonts}

\newcommand{\rb}[1]{\left( #1 \right)}
\newcommand{\bb}[1]{\left[ #1 \right]}
\newcommand{\cb}[1]{\left\{ #1 \right\}}

\usetikzlibrary{arrows.meta}
\usetikzlibrary{shapes,snakes}

\begin{document}

%	\title{Harmonic measure of scattered sets in $\BZ^d$, $d\ge3$}
	\title{Pool model: a mass preserving multi particle aggregation process}

	\author{Zhenhao Cai$^1$}
	\address[Zhenhao Cai]{Faculty of mathematics and computer science, Weizmann Institute of Science, Rehovot, Israel}
	\email{caizhenhao@pku.edu.cn}
	\thanks{$^1$Faculty of mathematics and computer science, Weizmann Institute of Science, Rehovot, Israel}

	\author{Eviatar B. Procaccia$^2$}
	\address[Eviatar B. Procaccia]{Faculty of Data and Decision Sciences, Technion - Israel Institute of Technology, Haifa, Israel}
	\email{eviatarp@technion.ac.il}
	\thanks{$^2$Faculty of Data and Decision Sciences, Technion - Israel Institute of Technology, Haifa, Israel}

    \author{Yuan Zhang$^3$}
	\address[Yuan Zhang]{Center for Applied Statistics and School of Statistics, Renmin University of China, Beijing, China}
	\email{zhang$\_$probab@ruc.edu.cn}
	\thanks{$^3$Center for Applied Statistics and School of Statistics, Renmin University of China, Beijing, China}

	\maketitle
	
	%\tableofcontents
	\begin{abstract}

We present and study the Pool model in $\BR^2$, a rotationally symmetric analogue of Multi-Particle Diffusion-Limited Aggregation (MDLA), in which particles (“droplets”) perform continuous-time random walks and are absorbed upon entering a circular pool initially centered at the origin. Each absorbed particle increases the pool’s mass, and the pool expands so that its area grows accordingly, yielding a natural mass-preserving dynamics. A central tool which is of independent interest is a version of Kurtz's theorem for this model, depicting the field of particles conditioned on the growth of the pool as an independent non-homogeneous Poisson point process. 
\end{abstract}
\setcounter{tocdepth}{2}
\tableofcontents
\section{Introduction}\label{Sec:intro}

Multi-Particle DLA was defined in \cite{meakin1988multiparticle,rosenstock1980cluster} as a model for electro-chemical deposition devoid of an external electric field. The model is defined by initiating i.i.d. Poisson$(\lambda)$ number of particles in each vertex of $\BZ^d$ each preforming continuous time random walk. When a random walk in some location tries to step on the aggregate, the location is added to the aggregate, and all other walkers in the location are annihilated. It is proven \cite{sidoravicius2019multi, sly2020one} that for high enough $\lambda$, the MDLA grows at a linear speed. However, it is conjectured \cite{sidoravicius2019multi} for $d\ge 2$ that the linear speed holds for any $\lambda$. This conjecture remains open, and understanding about the fractal dimension of the aggregate is expected to be essential. 

In this paper, we define and study the Pool model, in which droplets of unit mass diffuse and aggregate in a pool whose geometry is always rotationally invariant (and of full dimension). In this model, droplets which are instantaneously engulfed by the pool are not annihilated, and instead the model is mass preserving (unlike MDLA). We believe this is a good physical model for liquid pooling of diffusing droplets. The main result presented in this paper is three distinct phases. Mainly, diffusive regime for $\lambda<1$, non-explosive but faster than any polynomial sub-linear for $\lambda=1$, and explosion in finite time a.s. for $\lambda>1$. Thus, we see that the linear speed conjecture does not hold for the small intensity pool model, much like the 1d MDLA \cite{kesten2008problem}, where the aggregate grows in full dimension. However, though the aggregate has rotational symmetry, the fact that $d>1$ presents several challenges. First, as the pool grows, the extra radius each new particle adds decreases to zero, thus the distance in which one needs to look for the next particle changes with time, as well as the changing curvature of the pool creates a non-monotonic effect for the particle density (it is easier to hit smaller curvature). The particle density, conditioned on the growth, is controlled by a theorem attributed to private communication with Kurtz in \cite{kesten2008problem}. As far as we know, this paper is the first to put the proof in writing. Our proof may also be adapted to the version used in \cite{kesten2008problem}, to study MDLA. 

Recently, there have been some advancements in the understanding of multi-particle aggregation processes. The Cluster-Cluster model \cite{meakin1984diffusion} depicts diffusing clusters that aggregate upon impact. In 1d exact growth bound for cluster size is given in \cite{berger2025one}. Much like the engulfing pool model, in the Cluster-Cluster model there is an issue of explosion. This paper might present some insights for this problem for $d>1$. For off-lattice models, exact fractal dimension, growth rate and fluctuations are proven for the Stationary Hastings-Levitov$(0)$ model \cite{berger2022growth,berger2025logarithmic,chen2025one}. However, this model better resembles Stationary DLA rather than MDLA. Lastly we would like to highlight a related model called Finitary Random Interlacements, where there is a Poisson point process of finite random walk paths. In this model the condensation condition relates the density of the walks and their lengths \cite{bowen2019finitary, procaccia2019percolation, cai2021non}.

\subsection{Preliminaries}
%We define two models called additive pool model, denoted by $\CE_t\in\BR^+$, and annihilating pool model, denoted by $\CA_t\in\BR^+$, indicating the radii of the pools. We construct both along side the natural coupling between them. Let $A$ be a Poisson point process in $\BR^2$ of intensity $\lambda>0$. We start both processes with a pool of radius $\frac{1}{\sqrt{\pi}}$, $\CE_0=\CA_0=B_0(\frac{1}{\sqrt{\pi}})$. Next, on each point of $A$, place a Brownian particle, which holds labels $a_t,b_t\in\{0,1\}^{\BR^+}$. At each time if the a particle 
For a subset $A\subset \BR^2$, denote by $\text{Vol}(A)$ the area of $A$. Denote $B(r)=\{x\in\BR^2:|x|\le r\}$. 
Next, we formally define the model: Let $N(t)$ be a Poisson process of intensity $1$ in $\BR^2$. We call a process $\br_t$ continuous time and space random walk if $\br_t=\sum_{i=1}^{N(t)}\nu_i$, where $\nu_i\sim \text{Normal}\left(0, \left[\begin{array}{cc}
    1 & 0 \\
    0 & 1
\end{array}\right]\right)$. 
% We define a model called annihilating pool model, denoted by $\CA_t\in\BR^+$, indicating the radius of the pool as follows:  Let $A=\{X^i\}_{i=1}^\infty$ be a Poisson point process in $\BR^2$ of intensity $\lambda>0$. For each point $X^i$ of $A$, let $\br^i_t, \ t\ge 0$ be an independent continuous in time and space random walks starting from $X^i$. Let $A_t=\{\br^i_t\}_{i=1}^\infty$ be the configuration of the Poisson cloud at time $t$. For all $i$ and $t$ we associate a label $b_{i,t}\in \{0,1\}$ with the particle. We start the processes with a pool of radius $\frac{1}{\sqrt{\pi}}$, and set $b_{i,0}=\mathbbm{1}_{|X^i|>\CA_0}$. Now define stopping time 
% $$
% \tau_1=\inf\{t>0, |\br^i_t|\le \CA_0, \ b(i,t)=1\}
% .$$
% Then we set $\CA_t=\CA_0, \ b_{i,t}=b_{i,0}, \ \forall t\in [0,\tau_1)$, and let 
% $$
% \CA_{\tau_1}=\sqrt{\CA_0^2+\frac{1}{\pi}}
% .$$
% In other words, we add droplet of volume one to the pool and the radius changes accordingly. At the same time, the labels are updated so that particles which are now inside the pool are annihilated: $b_{i,\tau_1}=\mathbbm{1}_{|X^i|>\CA_{\tau_1}}$. Then recursively, we can define 
% \begin{equation}\label{eq:hittimesann}
% \tau_k=\inf\{t>\tau_{k-1}, |\br^i_t|\le \CA_{\tau_{k-1}}, \ b(i,t)=1\}
% \end{equation}
% and have $\CA_{\tau_k}=\sqrt{ \CA_{\tau_{k-1}}^2+\frac{1}{\pi}}$ together with $b_{i,\tau_k}=\mathbbm{1}_{|X^i|>\CA_{\tau_k}}$. 

% %%%%%%%%%%%%%%%%%%%%%%%%%%%%%%%%%%%%%%%%%%%
% %In the interest of comapring to discret models such as MDLA, we define the strongly annihilating model $\CS_t$.
% %%%%%%%%%%%%%%%%%%%%%%%%%%%%%%%

We define the pool model denoted by $\CE_t\in\BR^+$, indicating the radius of the pool as follows:  Let $A_0=\{X^i\}_{i=1}^\infty$ be a Poisson point process in $\BR^2$ of intensity $\lambda>0$. For each point $X^i$ of $A_0$, let $\br^i_t, \ t\ge 0$ be an independent continuous in time and space random walks starting from $X^i$. Let $A_t=\{\br^i_t\}_{i=1}^\infty$ be the configuration of the Poisson cloud at time $t$. For all $i$ and $t$ we associate a label $a(i,t)\in \{0,1\}$ with the particle. %We start the processes with a pool of radius $\frac{1}{\sqrt{\pi}}$,

In order to define the initial pool at time zero, consider the following procedure:
\begin{itemize}
\item Starting from radius $r_0=\frac{1}{\sqrt{\pi}}$, let $\xi_1=|A_0\cap B(r_0)|$.
\item Denote $r_1=\sqrt{\frac{1+\xi_1}{\pi}}$, and $\xi_2=|A_0\cap B(r_1)\setminus B(r_0)|$.
\item Repeat this process recursively, until recursion $N=\inf\{j: \xi_j=0\}$. If $N=\infty$, then we let $\CE_t\equiv \infty, \ \forall t\ge 0$, and say the process explodes at $t=0$. Otherwise, We let $\CE_0=r_{N}$ and set label
$$
a(i,0)=\mathbbm{1}_{\cb{|X^i|>\CE_0}}
.$$
\end{itemize}
Let $\hat \tau_0=0$, and for $k\ge 1$ when the model has not exploded,  define 
\bae\label{eq:engulfing_times}
\hat\tau_k=\inf\{t>\hat \tau_{k-1},\exists i\in A_0, |\br^i_t|\le \CE_{\hat \tau_{k-1}}, \ a(i,t)=1\}
,\eae
the first time after $\hat \tau_{k-1}$ such that an active particle runs into the pool, triggering the next round of growths. Then similarly we have 
\begin{itemize}
\item Starting from radius $r_0=\sqrt{\CE_{\hat \tau_{k-1}}^2+\frac{1}{\pi}}$, and $r_{-1}=\CE_{\hat \tau_{k-1}}$. Define $\xi_1=|A_{\hat \tau_k}\cap B(r_0)\setminus \bar B_0(r_{-1})|$;
\item Denote $r_1=\sqrt{r_0^2+\frac{\xi_1}{\pi}}$, and $\xi_2=|A_{\hat \tau_k}\cap B(r_1)\setminus B(r_0)|$.
\item Repeat this process recursively, until recursion $N(0)=\inf\{j: \xi_j=0\}$. If $N=\infty$, then we let $\CE_t\equiv \infty, \ \forall t\ge \hat \tau_k$, and say the process explodes at $t=\hat \tau_k$. Otherwise, We let $\CE_{\hat \tau_k}=r_{N}$ and set label
$$
a(i,\hat \tau_k)=\mathbbm{1}_{\cb{|\br^i_{\hat \tau_k}|>\CE_{\hat \tau_k}}}
.$$
\end{itemize}
Note that the pool model may explode in finite time, i.e. $\CE_t=\infty$ for some $t<\infty$. Denote by $T_{\mathfs{E}}$ the time of explosion: 
$$
 T_{\mathfs{E}}=\inf\{t: \ \CE_t=\infty\}
$$
%%%%%%%%%%%%%%%%%%%%%%%%%%%%%%%%%%%%%%%%%%
\subsection{Results}
First we state a version of a theorem that Kesten and Sidoravicius \cite{kesten2008problem} attribute to private communication with Kurtz, in the context of Multi-particle DLA. However, to our knowledge, no version of this proof appears in the literature. 
\begin{theorem}\label{thm:kurtz}
For every $t>0$, the random collection of points $$\mathfrak{K}_t=\{x\in A_t: |x|>\CE_t, a(x,t)=1\},$$ conditional on $\sigma\{\CE_s:s\le t\}$ is an independent Poisson point process, with intensity measure $\lambda \prob_x(|\br_s^x|>\CE_{t-s}, \forall s\le t)$.% i.e. the dependence to the aggregate is only through the intensity measure.
\end{theorem}
Next, we present an exact value for the critical density of instantaneous explosion. Moreover, we show that at the critical value $\lambda=1$, no explosion occurs, and in fact, in Theorem \ref{thm:growth}, we present a sublinear growth bound at criticality, and a diffusive growth bound at sub-criticality.   
\begin{theorem}[Explosion condition for engulfing pool model]\label{thm:explode}

\

\begin{enumerate}
\item For $\lambda>1$, $\{\CE_t\}_{t\ge0}$ explodes a.s. at a finite time.\\
\item For $\lambda=1$, $\{\CE_t\}_{t\ge0}$ does not explode a.s.
\end{enumerate}
\end{theorem}

\begin{theorem}[Growth rate]\label{thm:growth}

\

\begin{enumerate}

\item { For $\lambda<1$, we have asymptotically a.s. $\sqrt{t}\log^{-\frac{1+\epsilon}{2}}(t)\le \CE_t\le \sqrt{t}\log(t)$ for all $\epsilon>0$. }
\item For $\lambda=1$, $\forall\zeta<1$, {$\limsup_{t\to\infty}\frac{\CE_t}{t^\zeta}=\infty$ a.s.}
\end{enumerate}
\end{theorem}

\subsection{Open Problems}
\begin{enumerate}
    \item \label{open:critical} Theorem \ref{thm:growth} asserts that the Pool model at criticality grows faster than any sublinear power. The upper bound obtained in this paper (Corollary \ref{col_iterated_exp}) is very far from the lower bound. Based on computer simulations ({See Figure \ref{fig.cri}}),{ it seems that between big ``jumps" which correspond to an extraordinarily large Galton-Watson engulfing, possibly due to the local fluctuation of intensity, pool model $\CE_t$ at criticality seems to have a linear growth until it hits boundary of the box in our simulations. We anticipate that such big ``jumps"  will become increasingly rare as the aggregation grows.}  So we conjecture that 
    \begin{conjecture}
        At $\lambda=1$, there exists some $\xi>0$ such that $\lim_{t\to\infty}\frac{\CE_t}{t}= \xi$ a.s.
    \end{conjecture}
    \begin{figure}[H]
    \includegraphics[width=5.5in]{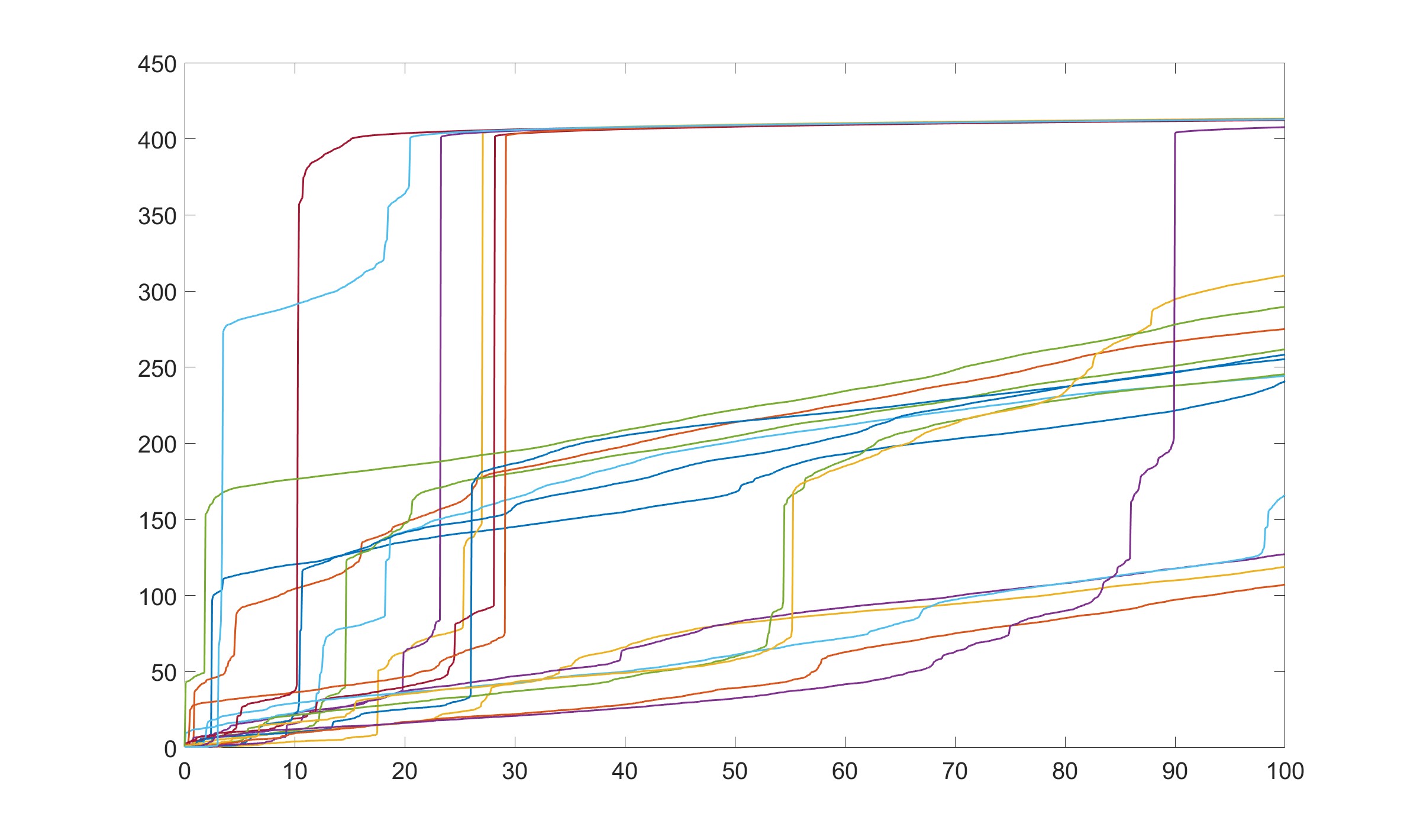}
    \caption{\label{fig.cri} Simulations on the growth of $\CE_t, t\in [0,100]$ in an approximated engulf pool model, where 1) free particles move in a finite box of size $800\times 800$ with periodic boundary conditions; and 2) engulfing is done at the end of each small deterministic time step $\delta_t=10^{-2}$. 20 random realizations are recorded.}
    \end{figure}

    \item \label{open:brownian} If one replaces in the model the continuous time and space random walks with Brownian motion, our proof of Theorem \ref{thm:explode} (2) fails. The reason is the speed up of Brownian motion in small time intervals, which causes problems whenever the particle density is high enough, since very close Brownian particles can enter the aggregate very fast. Thus we ask here: Does replacing continuous time and space random walk with Brownian motion lead to explosion at criticality $\lambda=1$?
    \begin{conjecture}
        The critical Brownian Pool model does not explode a.s.
    \end{conjecture}
    \item \label{open:annihilating} Much like the classical Multi-particle DLA, one can consider an annihilating version of the Pool model. The proof of Theorem \ref{thm:growth} (1) is very robust and would still hold. Thus we ask, what is the behaviour at criticality for the annihilating Pool model? We conjecture:
    \begin{conjecture}
        The annihilating Pool model at criticality grows at a linear speed.
    \end{conjecture}
    \item An engulfing version of MDLA can be considered by allowing excess particles at attached locations to preform a simple random walk inside the aggregate until hitting a vacant location on the boundary. This is consistent with some electrochemical deposition procedures \cite{paunovic2006fundamentals,vishnugopi2020surface}. Let the particle density be $\lambda$ and the internal random walk rates to be $\gamma$, with $\gamma=\infty$ corresponding to the limiting case where all excess particles are attached instantaneously.  
    \begin{conjecture}
        For $\gamma=\infty$, the engulfing MDLA if and only if $\lambda>1$. And for all large enough $\lambda$ and $\gamma>0$, the model grows linearly with its limiting shape converging to a ball as $\gamma\to\infty$.  
    \end{conjecture}%\note{I have revised the conjecture. Please take a look.--Yuan}
\end{enumerate}

\section{Kurtz's theorem}

In this section we prove Theorem \ref{thm:kurtz}. The proof follows by discretizing time and sequentially using the marking theorem \cite[Section 5.2]{kingman1992poisson}. The main difficulty of the proof is to encode the engulfing procedure in the discrete time intervals. 
\begin{proof}[Proof of Theorem \ref{thm:kurtz}]
We partition $[0,t]$ to $2^k$ intervals. Let 
\bae
&\mathfrak{I}_1=\left\{x\in A_0: \br^x_{[0,2^{-k}t]}\cap B(1/\sqrt{\pi})\neq\emptyset\right\}
%&\mathfrak{I}_1=\left\{x\in A_0: \br^x_{[0,2^{-k}t]}\cap B\left(\sqrt{\frac{1+\mathfrak{I}'_1}{\pi}}\right)\neq\emptyset\right\}
.\eae
For any $1<j<2^k$, %\note{Need to include the engulfing procedure. Idea 27.2.25: all the processes for different $k$s are coupled via the same $A_t$. After the first engulfing there is a non-infinitesimal gap between each arrival. As $k\to\infty$ get a.s. convergence of the processes in the lower time gaps after each real engulfing procedure}
\bae\label{eq:discretetimeprocess}
\mathfrak{I}_j=\left\{x\in A_0\setminus \bigcup_{i\le j-1}\mathfrak{I}_i : \br^x_{[(j-1)2^{-k}t,j2^{-k}t]}\cap B\left(\sqrt{\frac{1+\sum_{i=1}^{j-1}|\mathfrak{I}_i|}{\pi}}\right)\neq\emptyset \right\}
.\eae
By the Marking Theorem \cite[Section 5.2]{kingman1992poisson}, $\left\{\mathfrak{I}_j\right\}_{j=1}^{2^k}$ are conditionally independent Poisson Point Processes: Conditional on $\{\mathfrak{I}_i\}_{i=1}^{j-1}$,  $\mathfrak{I}_j$ is an independent PPP with intensity,
\bae\label{eq:conditional density}
\lambda \prob_x&\left( \br^x_{[(j-1)2^{-k}t,j2^{-k}t]}\cap B\left(\sqrt{\frac{1+\sum_{i=1}^{j-1}|\mathfrak{I}_i|}{\pi}}\right)\neq\emptyset , \right.\\
&\left.  \forall l<j, ~\br^x_{[(l-1)2^{-k}t,l2^{-k}t]}\cap B\left(\sqrt{\frac{1+\sum_{i=1}^{l-1}|\mathfrak{I}_i|}{\pi}}\right)=\emptyset  \right) 
.\eae
%For the case of the annihilating process $\CA_t$ \note{not annihilating}, since the probability that any Brownian motion hits the aggregate in the interval $[t-2^{-k},t]$ is summable as $k\to\infty$. By Borel-Cantelli, the process $\mathfrak{I}_{2^k}$ converges a.s, and in fact, there is a finite a.s. $K>0$ such that for all $k_1,k_2>K$, $\mathfrak{I}_{2^{k_1}}=\mathfrak{I}_{2^{k_2}}=\mathfrak{K}_t$. By construction, we obtain that the conditional distribution of $\mathfrak{K}_t$ is as stated. 

We need to take into account the engulfing procedure. The idea is to show that we can encode the instantaneous engulfing procedure in the first several iterations of the discrete time process \eqref{eq:discretetimeprocess}. Recall the notations of Section \ref{Sec:intro}, and in particular \eqref{eq:engulfing_times}, of the successive particle arrival times that induce the engulfing procedure $\hat\tau_k$. 

Assuming that $\CE_{0}<\infty$, $N<\infty$ (for $t=0$ engulfing procedure), and $\hat\tau_1>0$ is well defined. Take $k_0$, large enough such that
\begin{equation}\label{eq:k0cond}
   N 2^{-k_0}<\hat\tau_1. 
\end{equation} 
First we wish to take some $k_1>k_0$, such that all particles which are involved with the engulfing procedure, will make no movement in the time interval $[0,N2^{-k}t]$.
Define the event
\begin{equation}
\mathcal{S}_{k}:= 
\left\{
\forall x\in A_0 \cap B(\CE_0),\ 
\br^x_{[0,N2^{-k}t]}=x
\right\}.
\end{equation}
Conditional on $\CE_0$, there are exactly $\pi\CE_0^2$ particles in $B(\CE_0)$, so for $k$ large enough
\begin{equation}\label{eq:nomove}
    \prob(\mathcal{S}_{k}|\CE_0)=e^{-\pi\CE_0^2 N2^{-k}t}\ge 1-\frac{1}{2}\pi\CE_0^2 Nt2^{-k}.
\end{equation}
Thus, by Borel-Cantelli, conditional on $\CE_0$, there is a $k_1>k_0$ such that for any $k\ge k_1$, $\mathcal{S}_k$ occurs a.s. 

Next we want to make sure that particles outside $B(\CE_0)$ will not enter by time $N2^{-k}t$. Define the event
\[
\mathcal{G}_{k}:=
\left\{
\forall x\in A_0 \cap \bb{B(\CE_0)}^c,\ 
\br^x_{[0,N2^{-k}t]}\cap B(\CE_0)=\emptyset
\right\}.
\]
Take $\rho_k=2^{k/4}$, then by similar argument to \eqref{eq:nomove} we can find $k_2>k_1$ such that a.s. for all $k>k_2$, for all $x\in B(\CE_0+\rho_k)\setminus B(\CE_0)$, $\br^x_{[0,N2^{-k}t]}=x$. Moreover,
\begin{equation}
    \prob\left(\forall x\in A_0 \cap \bb{B(\CE_0+\rho_k)}^c,\ 
\br^x_{[0,N2^{-k}t]}\cap B(\CE_0)\neq\emptyset\right)\le c e^{-\rho_k}.
\end{equation}
Together we obtain that there is a $k_3>k_2$ such that a.s. for all $k>k_3$ both $\mathcal{S}_{k}$ and $\mathcal{G}_{k}$ occur. 

Now, for any $k>k_3$ we have by $\mathcal{S}_{k}$, that $A_0\cap B(r_0)\subset\mathfrak{I}_1$ and by $\mathcal{G}_{k}$ no other particles can enter, and thus $A_0\cap B(r_0)=\mathfrak{I}_1$. Similarly we obtain for any $k>k_3$ that for any $j\le N$, $A_0\cap B(r_j)\setminus B(r_{j-1})= \mathfrak{I}_j$, and $\sum_{i=1}^{N}|\mathfrak{I}_i|=\sum_{i=1}^N\xi_i$. Thus, if $t<\hat\tau_1$, there is some $l_0(k)$, $t\in[(l_0-1)2^{-k},l_02^{-k}]$. Then for some $k_4>k_3$ and any $k>k_4$ we have on an event of probability one, that  
\begin{equation}
   \cb{ \forall l<l_0(k), ~\br^x_{[(l-1)2^{-k}t,l2^{-k}t]}\cap B\left(\sqrt{\frac{1+\sum_{i=1}^{l-1}|\mathfrak{I}_i|}{\pi}}\right)=\emptyset}=\cb{|\br_s^x|>\CE_{t-s}, \forall s\le t}
,\end{equation}
then we are done. 

Otherwise, we obtain that for some $l_1(k)$, $\hat\tau_1\in[(l_1-1)2^{-k},l_12^{-k}]$, and assuming the engulfing procedure ends at a finite stage, for $k$ large enough $\hat\tau_2\notin[(l_1-1)2^{-k},(l_1+N)2^{-k}]$. By \eqref{eq:conditional density}, we obtain that conditional on $\{\mathfrak{I}_i\}_{i=1}^{l_1}$,  $\mathfrak{I}_{l_1+1}$ is an independent PPP with intensity,
\bae
\lambda \prob_x&\left( \br^x_{[l_1 2^{-k}t,(l_1+1)2^{-k}t]}\cap B\left(\sqrt{\frac{1+\sum_{i=1}^{l_1}|\mathfrak{I}_i|}{\pi}}\right)\neq\emptyset , \right.\\
&\left.  \forall l<l_1, ~\br^x_{[(l-1)2^{-k}t,l2^{-k}t]}\cap B\left(\sqrt{\frac{1+\sum_{i=1}^{l-1}|\mathfrak{I}_i|}{\pi}}\right)=\emptyset  \right) 
.\eae
By the strong Markov property of the random walk, we can now repeat the previous argument, by adjusting \eqref{eq:k0cond} to admit the number of engulfing steps in the time interval $\hat\tau_2-\hat\tau_1$. If all engulfing procedures end before time $t$, we are done. Otherwise we obtain that there are no active particles at time $t$, thus the statement follows trivially.  
\end{proof}
%%%%%%%%%%%%%%%%%%%%%%%%%%%%%%%%%%%%%%%%%%%%%%%%%%%%%%%%%%%%%%%%%%%%%%%%%%%%%%%%%%%%%%%%%%%%%%%%%%%%%%%%%%%%%%%%%%%%%%%%%%%%%%

The next corollary is a random stopping time version of Kurtz's theorem, whose proof is contained in the previous proof.
\begin{corollary}\label{cor:random_kurtz}
For every $k\in\BN$, the random collection of points $$\mathfrak{K}_{\hat\tau_k}=\{x\in A_{\hat\tau_k}: |x|>\CE_{\hat\tau_k}, a(x,{\hat\tau_k})=1\},$$ conditional on $\sigma\{\CE_s:s\le {\hat\tau_k}\}$ is an independent Poisson point process, with intensity measure $\lambda \prob_x(|\br_s|>\CE_{\tau_k-s}, \forall s\le {\hat\tau_k})$.
\end{corollary}

\begin{remark}
    Note that for future use include in Appendix \ref{sec:jurtz_brownian} a version of Kurt'z theorem, in the case the underlying particles are distributed as Brownian motions. See Open Problem \ref{open:brownian} for a Pool model driven by Brownian particles.
\end{remark}

%First if $A$ is a PPP over $S$ with intensity $\mu(x)$, then the distribution of $A$ given that $|A|>0$ is the same as adding one point of $S$ according to the distribution $\mu(x)/\mu(S)$, then sampling independently 

%%%%%%%%%%%%%%%%%%%%%%%%%%%%%%%%%%%%%%%%%%%%%%%%%%%%%%%%%%%%%%%%%%%%%%%%%%%%%%%%%%%%%%%%%%%%%%%%%%%%%%%%%%%%%%%%%%%%%%%%%%%%%%%%%%%%%%%%%%%%%%%%%%%%%%%%%%%%%%%%%%%%%%%%%%%%%%%%%%%%%%%%%%%%%%%%%%%%%%%%%

%%%%%%%%%%%%%%%%%%%%%%%%%%%%%%%%%%%%%%%%%%%%%%%%%%%%%%%%%%%%%%%%%%%%%%%%%%%%%%%%%%%%%%%%%%%%%%%%%%%%

\section{Explosion condition (Theorem \ref{thm:explode} (1))}
To prove Theorem \ref{thm:explode}, we first present a sequence of auxiliary lemmas. Denote by $\CE_{0,r_0}$ the aggregation we get from the initial engulfing procedure at $t=0$ with initial radius $r_0\ge \frac{1}{\sqrt{\pi}}$. 

\begin{lemma}
	\label{lem_exp}
	For all $\lambda>1$, there is some $c=c(\lambda)>0$ such that for all $r_0\ge\frac{1}{\sqrt{\pi}}$, 
	$$
	\prob(\CE_{0,r_0}=\infty)\ge c>0.
	$$
\end{lemma}

\begin{proof}
	By the definition of Poisson point process $A$. It is immediate that $\xi_1$ is a Poisson random variable with intensity $\lambda \pi r_0^2> 1$. And given $\xi_i=k>0$, $\xi_{i+1}\sim \text{Pois}(k\lambda)$. Thus $\xi_i$ is identically distributed as the size of the $i$th generation in a supercritical Galton-Watson tree with offspring distribution $\text{Pois}(\lambda)$. Denote by $p_0(\lambda)>0$ the surviving probability of the G-W process. Then 
	$$
	\prob(\CE_{0,r_0}=\infty)\ \ge (1-e^{-1})p_0(\lambda):=c(\lambda).
	$$ 
\end{proof}

The next lemma shows that given a Poisson cloud of particles outside a ball, one can always wait for long enough so that the particles will move into the vacuum as much as possible. To be precise,  for all $\lambda>0$ and $R\ge \frac{1}{\sqrt{\pi}}$, let $A^{\lambda,R}=\{Y^i\}$ be a Poisson point process supported on $B(R)^c$ with intensity $\lambda$. For each point $Y^i$ of $A$, let $\br^i_t, \ t\ge 0$ be an independent continuous in time and space random walks starting from $Y^i$. Let $A^{\lambda,R}_t=\{\br^i_t\}_{i=1}^\infty$ be the configuration of the Poisson cloud at time $t$. We denote by $\Lambda^{\lambda,R}_t(x), x\in \BR^2$ the intensity of the aforementioned process at time $t$.

\begin{lemma}
	\label{lem_refill}
	For all $\delta>0$, there is an $M=M(\delta)<\infty$ such that for all $\lambda>0$ and $R\ge \frac{1}{\sqrt{\pi}}$,
	$$
	\Lambda^{\lambda,R}_t(x)\ge \lambda(1-\delta), \ \forall x\in \BR^2
	$$
	for all $t\ge M+R^3$. 
\end{lemma}

\begin{proof}
	By reversibility, for any $x\in \BR^2$, 
	$$
	\Lambda^{\lambda,R}_t(x)=\lambda\prob_x(X_t\notin B(R))=\lambda\rb{1-\prob_x(X_t\in B(R))},
	$$
	where 
	$$
	X_t=\sum_{i=1}^{N_t} Z_i
	$$ 
	with $\{Z_i\}_{i=1}^\infty$ an i.i.d. sequence of $N(0,I_2)$ r.v.'s and $N_t, t\ge 0$ a standard process independent to $\{Z_i\}$. For fixed $\delta>0$, there is an $M_1$ such that for all $t\ge M_1$, 
	$$
	\prob\left(N_t\in \left[\frac{t}{2}, 2t\right]\right)\ge 1-\frac{\delta}{2},
	$$
	implying that 
	$$
	\prob_x(X_t\in B(R))\le \frac{\delta}{2}+\max_{s\in [\frac{t}{2}, 2t]} \prob_x(N(0,sI_2)\in B(R)). 
	$$
	Now noting that
	$$
	\prob_x(N(0,sI_2)\in B(R))= \prob_x\rb{N(0,I_2)\in B\rb{\frac{R}{\sqrt{s}}}}\le \frac{\pi R^2}{s}\times \frac{1}{2\pi},
	$$
	we have for $s\in  [\frac{t}{2}, 2t]$ and $t\ge M+R^3$
	\beq
	\label{eq_Norm_pdf}
	\prob_x(N(0,sI_2)\in B(R))\le \frac{R^2}{M+R^3}.
	\eeq
	So there is a $M_2$ such that $\frac{R^2}{M+R^3}<\frac{\delta}{2}$ for all $M\ge M_2$ and $R\ge 0$, which conclude the proof of this lemma.
\end{proof}

Now we can introduce a increasing sequence of random times, denoted by $\{\tau_n, n=1,2,\cdots\}$ each of which will be later shown to give an uniformly bounded chance of explosion. 

\begin{itemize}
	\item If $\CE_0=\infty$, we have an explosion at the very beginning and can simply let $\tau_n\equiv \infty$ for all $n\ge 1$.
	\item Given $\CE_0<\infty$, let $R_1=\CE_0$, $\tau_1=M+R_1^3$;
	\item Now for all integer $k\ge 1$, if $\CE_{\tau_k}=\infty$, then let $\tau_j\equiv \infty$ for all $j\ge k+1$; and if $\CE_{\tau_k}<\infty$, let $R_{k+1}=\CE_{\tau_k}+\tau_k+M_3$, and $\tau_{k+1}=\tau_k+R_{k+1}^3+M$, where $M_3>0$ is some absolute constant to be specified later. 
\end{itemize}
With the construction above, we further have:
\begin{lemma}
	\label{lem_stop}
	$\{\tau_n, n=1,2,\cdots\}$ form a sequence of stopping times with respect to filtration $\sigma\{\CE_s, 0\le s\le t\}$, each of which is supported on a discrete set $I_n$ such that for all $T<\infty$, $|I_n\cap (0,T]|<\infty$. 
\end{lemma}	

\begin{proof}
	We prove this lemma by induction. Note that the engulfing process $\{\CE_t\}_t$ always has integer volume, which implies that the first stopping time $\tau_1\in \sigma(\CE_0)$ satisfies the lemma. Now assume this holds true for $k\ge 1$. Then for each $T<\infty$, enumerate the finite set $I_k\cap (0,T]$ by $\{t_1,\cdots, t_m\}$. Now for $\tau_{k+1}$, it is easy to see that 
	$$
	\{\tau_{k+1}<T\}=\bigcup_{\scriptsize\begin{aligned}&j\le m, (i/\pi)^{3/2}<T,\\ &t_j+[(i/\pi)^{1/2}+t_j+M_3]^3+M<T\end{aligned}} \{\tau_{k}=t_j\}\cap \{\CE_{t_j}=(i/\pi)^{1/2}\}.
	$$
	Since the union above is finite, and for each pair of $(i,j)$, $\tau_{k+1}\equiv t_j+[(i/\pi)^{1/2}+t_j+M_3]^3+M $ under event $ \{\tau_{k}=t_j\}\cap \{\CE_{t_j}=(i/\pi)^{1/2}\}$, one has verified that $|I_{k+1}\cap (0,T]|<\infty$. Moreover, recalling that $\tau_n$ is a stopping time, 
	$$
	\{\tau_{k}=t_j\}\cap \{\CE_{t_j}=(i/\pi)^{1/2}\}\in \sigma\{\CE_s, 0\le s\le T\}
	$$
	which concludes the proof of this lemma. 
\end{proof}

Recalling the definition of $A_t$, at let $\tilde A_n:=\{\tilde x_{n,j}\}_{j=1}^\infty$ be the collection of $\br^i_{\tau_n}$ with $a(j,\tau_n)=1$. I.e. $\tilde A_n$ is the field of particles outside $\CE_{\tau_n}$ that has not yet released their masses. Given Lemma \ref{lem_stop} and Theorem \ref{thm:kurtz}, one may see that: for any $n\ge 1$ under event $\CE_{\tau_n}<\infty$ the conditional distribution of $\tilde A_n$ given $\{\CE_s, 0\le s\le\tau_n\}$ is a Poisson point process with intensity given by 
\beq
\label{eq_con_inten}
\tilde \Lambda(x|\CE_s, 0\le s\le\tau_n)=\lambda \prob_x(|X_s|>\CE_{\tau_n-s}, \ \forall s\le \tau_n).
\eeq
The subsequent lemma shows that this conditional intensity is virtually intact beyond radius $R_n$. 

\begin{lemma}
	\label{lem_int_far}
	For all $\delta>0$, there is a constant $M_3$ such that for all $n$, we a.s. have 
	$$
	\tilde \Lambda(x|\CE_s, 0\le s\le\tau_n)\ge \lambda (1-\delta), \forall x\in B(R_{n+1})^c
	$$
	under event $\{\CE_{\tau_n}<\infty\}$. 
\end{lemma}

\begin{proof}
	Recalling that $R_{n+1}=\CE_{\tau_n}+\tau_n+M_3$, by \eqref{eq_con_inten}, we may treat $\tau_n$ and the growth history of $\CE$ as fixed and have 
	\begin{align*}
		\tilde \Lambda(x|\CE_s, 0\le s\le\tau_n)&=\lambda\prob_x(|X_t|>\CE_{\tau_n-t}, \ \forall t\le \tau_n)\\
		&\ge \lambda \prob_x(|X_t|>\CE_{\tau_n}, \ \forall t\le \tau_n)\\
		&\ge \lambda \prob_x(|X_t-x|<\tau_n+M_3, \ \forall t\le \tau_n).
	\end{align*}
	Since it is easy to see that $\prob_0(\max_{s\le 2t}|B_s|\ge t)\to 0$ as $t\to\infty$. one can choose a finite $M_3$ such that 
	$$
	\prob_0(\max_{s\le 2t}|B_s|\ge t+M_3/2)<\frac{\delta}{2}, \forall t\ge 0
	$$
	while at the same time 
	$$
	\prob(N_t\ge 2t+M_3)<\frac{\delta}{2}, \forall t\ge 0
	$$
	where $N_t$ is a standard Poisson process. Thus by the definition of $X_t$, 
	\begin{align*}
		\prob_x(|X_t-x|<\tau_n+M_3, \ \forall t\le \tau_n)\ge 1&-\prob(N_{\tau_n}\ge 2\tau_n+M_3)\\
		&-\prob_0(\max_{s\le 2\tau_n+M_3}|B_s|\ge \tau_n+M_3/2+M_3/2)\ge 1-\delta.
	\end{align*}
\end{proof}

Now under event $\{\CE_{\tau_n}<\infty\}$, we can retain the mass of particles in $\tilde A_n$ and release them all at once at $\tau_{n+1}$. To be precise, 

\begin{itemize}
	\item Let $r_{n,0}=\CE_{\tau_n}$, $\hat A_{n,0}=\{i: a(i,\tau_n)=1, \ \br^i_{\tau_{n+1}}\in B_0(\CE_{\tau_n})\}$, $\xi_{n,0}=|\hat A_{n,0}|$ and let $r_{n,1}=\sqrt{\CE_{\tau_n}^2+\frac{\xi_{n,0}}{\pi}}$ if $\xi_{n,0}>0$
	\item Then for each layer $m\ge 1$, $\hat A_{n,m}=\{i: a(i,\tau_n)=1, \ \br^i_{\tau_{n+1}}\in B_0(r_{n,m})\setminus B_0(r_{n,m-1})\}$ and $\xi_{n,m}=|\hat A_{n,m}|$, let $r_{n,m+1}=\sqrt{r_{n,m}^2+\frac{\xi_{n,m}}{\pi}}$. 
\end{itemize}
We say an delayed explosion happens at step $n$ if $\xi_{n,m}>0$ for all $m$, and denote such event by $\chi_n$. The following lemma shows that delaying the release of mass can only slow down the growth of aggregation. 

\begin{lemma}
	\label{lem_retain}
	For all integers $n,m$ and $i\in \hat A_{n,m}$, one always has $a(i,\tau_{n+1})=0$, i.e., the particle has released its mass by the next stopping time. 
\end{lemma}	

\begin{proof}
	We prove this by induction on layers of the engulfing procedure. By definition, for each $i$, $a(i,\tau_{n+1})=1$ if and only if $|\br^i_t|>\CE_t, \ \forall t\le \tau_{n+1}$. So for each $i\in \hat A_{n,0}$, its mass must have been released since $\CE$ is increasing, thus 
	$$
	\CE_{\tau_{n+1}}\ge \sqrt{\CE_{\tau_n}^2+\frac{\xi_{n,0}}{\pi}}=r_{n,1}.
	$$
	Now by induction, if we assume for some $m\ge 1$, %\note{changed $n$ to $m$}
	$$
	a(i,\tau_{n+1})=0, \ \forall i\in \bigcup_{j=0}^{m-1} \hat A_{n,j},
	$$
	then by volume argument one must have 
	$$
	\CE_{\tau_{n+1}}\ge\sqrt{\CE_{\tau_n}^2+\frac{\sum_{j=0}^{m-1}\xi_{n,j}}{\pi}}=r_{n,m}.
	$$
	Now for each particle $i\in \hat A_{n,m}$, 
	$$
	\br^i_{\tau_{n+1}}\in B_0(r_{n,m})\setminus B_0(r_{n,m-1})\Rightarrow \left| \br^i_{ \tau_{n+1}}\right|<\CE_{\tau_{n+1}}
	$$	
	implying that $a(i,\tau_{n+1})=0$ and that all the $\xi_{n,m}$ particles in $\hat A_{n,m}$ have also released their masses, which gives 
	$$
	\CE_{\tau_{n+1}}\ge\sqrt{r_{n,m}^2+\frac{\xi_{n,m}}{\pi}}=r_{n,m+1}. 
	$$
	Thus, all particles have released their masses by no later than $\tau_{n+1}$. 
\end{proof}

\begin{proof}[Proof of Theorem \ref{thm:explode}]
	Now we are ready to finish the proof of explosion of the supercritical pool model. Recalling the definition of our stopping times, it is easy to see that $\{T_\CE=\infty\}\subset \{\CE_{\tau_n}<\infty\}$ for all $n\ge 1$. Thus it suffices to prove that for all $\lambda>1$, there is a $c=c(\lambda)\in (0,1)$ such that
	\beq
	\label{eq_geo}
	\prob(\CE_{\tau_{n+1}}<\infty| \CE_{\tau_n}<\infty)\le 1-c\Rightarrow \prob(\CE_{\tau_n}<\infty)\le (1-c)^{n-1}
	\eeq
	To see this, by Lemma \ref{lem_retain}, we only need to prove 
	$$
	c\le \prob(\chi_n| \CE_{\tau_n}<\infty)\le \prob(\CE_{\tau_{n+1}}=\infty| \CE_{\tau_n}<\infty).
	$$
	Now under event $\{\CE_{\tau_n}<\infty\}$, and given the growth history $\{\CE_s, 0\le s\le\tau_n\}$, applying Lemma \ref{lem_int_far} and \ref{lem_refill} with $\delta=(\lambda-1)/3$, we have the particles in $\tilde A_n$ (conditionally) form a Poisson point process at $\tau_{n+1}$ with intensity uniformly bounded from below by $\lambda(1-\delta)^2> \frac{\lambda+2}{3}>1$ for all $x\in \BR^2$. Thus by Lemma \ref{lem_exp}, there is a $c=c(\lambda)>0$ such that we have 
	$$
	\prob(\chi_n|\CE_s, 0\le s\le\tau_n)\ge c
	$$
	holds almost surely under event $\{\CE_{\tau_n}<\infty\}$, and a total probability theorem concludes the proof. 
\end{proof}

\section{Critical case $\lambda=1$ {(Theorem \ref{thm:explode} (2))}
}In this section we first prove the second half of Theorem \ref{thm:explode}, i.e. the engulfing process $\CE_t$ does not explode at $\lambda=1$. To see this, one may first note that by Theorem \ref{thm:kurtz} the intensity of active particles with $a(\cdot,t)=1$ at each fixed time $t$ as well as stopping time $\hat\tau_k$ is bounded from above by 1. This also implies that for each round of engulfing procedure, the Galton-Watson process is stochastically bounded from above by the critical one with offspring distribution $\text{Pois}(1)$. With this observation, we can first consider the case where active particles is of density 1, and present the following result which gives a lower bound on the waiting time between arrivals. For any $R>0$. Let $\tilde A^R=\{\tilde X^i\}_{i=1}^\infty$ be a Poisson point process supported on $B(R)^c$ with intensity 1.  For each point $\tilde X^i$ of $A$, let $\tilde\br^i_t, \ t\ge 0$ be the same independent continuous in time and space random walks starting from $\tilde X^i$. And denote by $\tilde A^R_t$ the configuration of particles $\{\tilde\br^i_t, \ i\ge 0\}$. 

\begin{lemma}\label{lem:inf_rate}
	Consider stopping time 
	$$
	\Delta\tau=\inf\{t>0, \tilde A^R_t\cap B(R)\not=\emptyset \}
	$$
	be the first time that a particle enters the ball of radius $R$. Then there is some constant $C<\infty$ such that for all $R\ge \pi^{-1/2}$ and $t\ge 0$, we have 
	\bae
	\label{eq_waiting}
	r(R,t):=\lim_{\Delta t\to0}\frac{P(\Delta\tau\in (t,t+\Delta t]|\Delta\tau>t)}{\Delta t}\le r(R,0)\le CR. 
	\eae
\end{lemma}

\begin{proof}
	To prove this lemma, we first note that by Theorem \ref{thm:kurtz}, for any $t>0$ given $\Delta\tau>t$ then intensity of free particles at $x\in B^c(R)$ is given by	$\prob\left(\br^x_{[0,t]}\cap B(R)=\emptyset\right)\le 1$. Thus the conditional distribution of $\Delta\tau-t$ stochastically dominates $\Delta\tau$ and  
	$$
	r(R,t)=\lim_{\Delta t\to0}\frac{P(\Delta\tau\in (t,t+\Delta t]|\Delta\tau>t)}{\Delta t}\le\lim_{\Delta t\to0}\frac{P(\Delta\tau\in (0,\Delta t])}{\Delta t}= r(R,0)
	$$
	which verifies the first inequality of \eqref{eq_waiting}. 
	
	Now to bound $r(R,0)$ from above, we then show that the probability contributed by particles starting from at least $\Delta t^{-1/3}$ away from $B(R)$ is of a higher order. To be precise, we assert that
	\bae
	\label{eq_con_high}
	\prob(E_{R,\Delta t}):=\prob\left(\exists i \text{ s.t. } |\tilde X^i|\ge R+\Delta t^{-1/3}, \ \tilde\br^i_{[0,\Delta t]}\cap B(R)\not=\emptyset \right)=s.e(\Delta t^{-1})R.  
	\eae
	It is immediate to see that the LHS of \eqref{eq_con_high} can be bounded from above by the corresponding expectation:
	\begin{align*}
	&\sum_{i=1}^\infty 	\prob\left(|\tilde X^i|\ge R+\Delta t^{-1/3}, \ \tilde\br^i_{[0,\Delta t]}\cap B(R)\not=\emptyset \right)\\
	&\le  \sum_{n=\lfloor \Delta t^{-1/3}\rfloor}^\infty \sum_{i=1}^\infty 	\prob\left(|\tilde X^i|\in [R+n,R+n+1], \ \tilde\br^i_{[0,\Delta t]}\cap B(R)\not=\emptyset \right)
	\end{align*}
	Then for each integer $n\ge \lfloor \Delta t^{-1/3}\rfloor$ and $|x_0|\in [R+n,R+n+1]$, recalling the definition of our continuous random walk $\br_{\cdot}$, one has 
	\bae
	\begin{aligned}
		\prob\left(\tilde\br^{x_0}_{[0,\Delta t]}\cap B(R)\not=\emptyset \right)&\le \prob(\text{Pois}(\Delta t)\ge n)+ \prob(|N(0,n)|\ge n/2)\\
		&\le C (\Delta t)^n+ \exp(-c n)\le 2C  \exp(-c n).
	\end{aligned}	
	\eae
	Then note that the expected number of particles of $\tilde A^R$ in annulus $\bar B(R+n+1)\setminus B(R+n)$ equals to its volume which is 
	$$
	\Delta V_n=\pi(R+n+1)^2-\pi(R+n)^2\le C(R+n).
	$$
	Thus we have for each $n$
	\begin{align*}
	\sum_{i=1}^\infty 	\prob\left(|\tilde X^i|\in [R+n,R+n+1], \ \tilde\br^i_{[0,\Delta t]}\cap B(R)\not=\emptyset \right)\le C(R+n)\exp(-cn)
	\end{align*}
	which implies that
	\begin{align*}
		&\sum_{i=1}^\infty 	\prob\left(|\tilde X^i|\ge R+\Delta t^{-1/3}, \ \tilde\br^i_{[0,\Delta t]}\cap B(R)\not=\emptyset \right)\\
		&\le  CR\sum_{n=\lfloor \Delta t^{-1/3}\rfloor}^\infty \exp(-cn)+\sum_{n=\lfloor \Delta t^{-1/3}\rfloor}^\infty \exp(-cn/2)\le CR \exp(-c\Delta t^{-1/3}/3)
	\end{align*}
	and thus validate \eqref{eq_con_high}. Now noting that 
	\bae
	\label{eq_63}
	\begin{aligned}
	P(\Delta\tau\in (0,\Delta t])&\le\prob\left(\exists i \text{ s.t. } R<|\tilde X^i|< R+\Delta t^{-1/3}, \ \tilde\br^i_{[0,\Delta t]}\cap B(R)\not=\emptyset \right)+\prob(E_{R,\Delta t})\\
	&\le \prob\left(\exists i \text{ s.t. } R<|\tilde X^i|< R+\Delta t^{-1/3}, \ \tilde\br^i_{[0,\Delta t]}\cap B(R)\not=\emptyset \right)+s.e(\Delta t^{-1})R,
	\end{aligned}
	\eae
	and for the first term, we define events
	$$
	E_1:=\left\{\exists i \text{ s.t. } R<|\tilde X^i|< R+\Delta t^{-1/3}, \ \text{ and } \tilde\br^i_{[0,\Delta t]} \text{ has at least two jumps} \right\}, 
	$$
	and
	$$
	E_2:=\left\{\exists i \text{ s.t. } R<|\tilde X^i|< R+\Delta t^{-1/3}, \ \sigma_{i,1}<\Delta t, \  \tilde\br^{i,em}_1\in B(R) \right\}, 
	$$
	where $\sigma_{i,1}$ and $\tilde\br^{i,em}_\cdot$ is respectively first transition time and the embedded Markov Chain of $\tilde\br^i_\cdot$. For event $E_1$, it is easy to see that for each $i$
	$$
	\prob\left( \tilde\br^i_{[0,\Delta t]} \text{ has at least two jumps}\right)\le C \Delta t^2,
	$$
	while 
	$$
	N_1=|\{i: \ R<|\tilde X^i|< R+\Delta t^{-1/3}\}|\sim \text{Pois}\left(\pi( R+\Delta t^{-1/3})^2-\pi R^2 \right),
	$$
	which together imply that 
	\bae
	\label{eq_E1}
	P(E_1)\le  C \Delta t^2\left[ \pi( R+\Delta t^{-1/3})^2-\pi R^2  \right]\le C(R\Delta t^{5/3}+\Delta t^{4/3})=Ro(\Delta t).
	\eae
	Finally, to bound the probability of event $E_2$, one may for each integer $n\in [0,\Delta t^{-1/3}]$ define 
	$$
	E_{2,n}:=\left\{\exists i \text{ s.t. } R+n<|\tilde X^i|< R+n+1, \ \sigma_{i,1}<\Delta t, \ \tilde\br^{i,em}_1\in B(R) \right\}.
	$$
	Then one may immediately see that $\prob(E_2)\le \sum_{n=1}^{\lfloor\Delta t^{-1/3} \rfloor}\prob(E_{2,n})$, and that for each such $n$
	$$
	\prob(E_{2,n})\le [\pi(R+n+1)^2-\pi(R+n)^2] [1-\exp(\Delta t)] \max_{x\in B(R+n+1)\setminus B(R+n)}\prob\left(x+\mathcal{W}_1\in B(R)\right)
	$$
	where $\mathcal{W}_t$ is a standard BM. Moreover by a standard estimation of Gaussian kernel,
	$$
	 \max_{x\in B(R+n+1)\setminus B(R+n)}\prob\left(x+\mathcal{W}_1\in B(R)\right)\le \prob(|\mathcal{W}_1|\ge n)\le \exp(-c n^2)
	$$ 
	for some $c>0$. Thus we have 
	$$
		\prob(E_{2,n})\le C(R+n)\exp(-cn^2)\Delta t
	$$
	which in turn implies that 
	\bae
	\label{eq_E2}
	\prob(E_2)\le \sum_{n=1}^{\lfloor\Delta t^{-1/3} \rfloor}\prob(E_{2,n})\le (C_1 R+C_2)\Delta t.
	\eae
	Combining \eqref{eq_63}-\eqref{eq_E2}, we conclude the proof of this lemma. 
\end{proof}

As a direct corollary of Lemma \ref{lem:inf_rate}, we also have that 
\begin{corollary}
	\label{col:exp_dom_waiting}
	 For all $R\ge \pi^{-1/2}$ and random field $\tilde A_R$, stopping time $\Delta \tau$ defined above stochastically dominates an exponential random variable with exception $1/(CR)$.  
\end{corollary}

With the corollary above and Theorem \ref{thm:kurtz}, one can now upper bound each round of engulfing procedure by i.i.d. critical ones, while at the same time lower bound each waiting time between arrivals by exponential r.v.'s, and construct the following process which stochastically dominates our engulfing pool model at criticality.  

\begin{itemize}
	\item Let $\{X_n\}_{n=0}^\infty$ be an i.i.d. sequence of random variables with the distribution of the total progeny of a critical GW tree with offspring distribution $\text{Pois}(1)$. 
	
	\item Let $\{T_n\}_{n=0}^\infty$ be an i.i.d. sequence of random exponential variables with $\lambda=1$. 
	
	\item For each $n\ge 0$, let $R_n=\sqrt{\frac{\sum_{i=0}^n X_i}{\pi}}$, and 
	$$
	\tau_0=0, \ \tau_n=\sum_{k=0}^{n-1} \frac{T_k}{C R_k}. \ \forall n\ge 1,
	$$
	where $C$ is the constant in Corollary \ref{col:exp_dom_waiting}. Now for each $t\ge 0$, let $n_t=\sup\{n: \ \tau_n\le t\}$, and let $\tilde \CE_t=R_{n_t}$. 
\end{itemize}

{By Corollary \ref{col:exp_dom_waiting} and Theorem \ref{thm:kurtz}, it is straightforward to see that $\tilde \CE_t\gtrsim \CE_t$.} Note that given $\{X_n\}_{n=1}^\infty$ , the distribution of $\frac{T_n}{R_n}, \ n\ge 1$ is an sequence of independent exponential random variables with $\lambda_n=R_n$, and that $\tilde \CE_t<\infty, \forall t<\infty$ is equivalent to $\tau_n\to \infty$. The following lemma shows that $\tau_n$ is a.s. equivalent to $\sum_{k=0}^{n-1} \frac{1}{R_k}$. 

\begin{lemma}
	\label{lem:equi}
	Let $T_n, \ n\ge 1$ be a sequence of independent exponential random variables with $E[T_n]=c_n$ such that $c_1=\pi^{1/2}$, $c_n\downarrow 0$, while $\sum_{n=1}^\infty c_n=\infty$. Then we have 
	$$
	\frac{\sum_{k=1}^n T_k}{\sum_{k=1}^n c_k}\overset{a.s.}{\rightarrow} 1. 
	$$
\end{lemma}

\begin{proof}
	For each integer $N$, define
	$$
	K_N=\inf\left\{k: \sum_{i=1}^k c_i\ge N^4\right\}.
	$$
	By definition, it is immediate to see that $K_N<K_{N+1}$ for all sufficiently large $N$. Then noting that $\var(T_k)=c_k^2\le \pi^{1/2}c_k$, by Chebyshev inequality, 
	\bae
	\label{Ch_1}
	\prob\left(\left|\frac{\sum_{j=1}^{K_N} T_j}{\sum_{j=1}^{K_N} c_j}-1\right|\ge \frac{1}{N} \right)\le \frac{N^2\sum_{j=1}^{K_N} c_j^2}{(\sum_{j=1}^{K_N} c_j)^2}\le \frac{\pi^{1/2}}{N^2}
	\eae
	which is summable. Thus, we have
	$$
	Y_N=\frac{\sum_{j=1}^{K_N} T_j}{\sum_{j=1}^{K_N} c_j}\overset{a.s.}{\rightarrow} 1. 
	$$
	At the same time for all $K_N<n<K_{N+1}$ we have
	$$
	\frac{Y_N N^4}{(N+1)^4}\le \frac{\sum_{k=1}^n T_k}{\sum_{k=1}^n c_k}\le \frac{Y_{N+1}[(N+1)^4+\pi^{1/2}]}{N^4}
	$$
	which concludes the proof of this lemma. 
\end{proof}

With Lemma \ref{lem:equi}, to prove the desired non-explosion, it suffices to show the a.s. divergence of $\sum_{k=0}^{n-1} \frac{1}{R_k}$. 

\

\

%Using \cite{otter1949multiplicative,dwass1969total,bingham1987regular,janson2012simply}
\begin{lemma}
	Let $X_1,X_2,\dots$ be i.i.d.\ copies of the total progeny of a critical Galton--Watson tree where the number of offspring is distributed as $\text{Pois}(1)$. Set $S_n:=X_1+\cdots+X_n$. Then
	\[
	\sum_{n=1}^\infty \frac{1}{\sqrt{S_n}}=\infty \qquad \text{almost surely}.
	\]
\end{lemma}

\begin{proof}
	Note that by classic theory of Galton-Watson process,  e.g. \cite{dwass1969total,bingham1987regular}, the total number of progeny of such a critical GW process is distributed as follows:
	$$
	\prob(X_1=n)=\frac{\prob\left(\text{Pois}(n-1)=n-2\right)}{n-1}=\frac{e^{-(n-1)}(n-1)^{-(n-2)}}{(n-1)!}.
	$$
	Then by Stirling's formula, $\prob(X_1=n)\sim C n^{-3/2}$, which implies that 
	\bae
	\label{stable_law_p}
	\prob(X_i\ge n)=\sum_{k=n}^\infty \prob(X_1=k)\sim C n^{-1/2}.
	\eae
	Now for $a_n=\inf\{x: \ \prob(|X_1|> x)\le \frac{1}{n}\}$, and $b_n=n E[X_1 \ind_{|X_1|\le a_n}]$, it is easy to see that $a_n\sim C n^2$, while 
	$$
	b_n=n\sum_{k=1}^{a_n} k\prob(X_1=k)\le =n\sum_{k=1}^{(C+\epsilon)n^2} k\times (C+\epsilon) k^{-3/2}\le (C+\epsilon)^{3/2} n^2
	$$
	for all sufficiently large $n$. Similarly we also have $b_n\ge (C-\epsilon)^{3/2} n^2$ for all sufficiently large $n$. Thus by Theorem 3.8.2 of \cite{durret2019probability}, we have $\frac{S_n}{n^2}\Rightarrow Y$ where $Y\ge 0$ is some non-degenerate random variable. Thus there is some $M<\infty$ and $p_0>0$ independent to $n$ such that
	$$
	\prob\left(\frac{S_n}{n^2}\le M \right)\ge p_0
	$$
	for all integer $n\ge 1$.
	
	Now noting that $X_i$'s are all positive which guarantees that $1/\sqrt{S_n}$ is decreasing, it suffices to show 
	\bae
	\label{eq_condensation}
	\sum_{n=1}^\infty M_n:=\sum_{n=1}^\infty \frac{2^n}{\sqrt{S_{2^n}}}=\infty,
	\eae
	where we use the convention $M_0=0$. And in order to show \eqref{eq_condensation}, one only needs to prove that there is some constant $c>0$ such that with probability one $M_n\ge c$ infinitely often. To see this, we first note that for any $\delta\in (0,1)$ and integer $n$,
	$$
	\prob\left(S_{2^n}\ge 2^{(3+\delta)n}\right)\le \sum_{k=1}^{2^n} \prob\left(X_k\ge 2^{(2+\delta)n}\right)\le C2^{-\delta n/2}
	$$
	which is summable. Thus, we have with probability one $S_{2^n}< 2^{(3+\delta)n}$ for all sufficiently large $n$. We denote such event by $E$.

	Now for subseqeunce $k_n=2^n$ (with $k_0=0$), define a sequence of independent random variables
	$$
	L_n=\frac{\sum_{j=2^{k_{n-1}}+1}^{2^{k_n}} X_j}{2^{2k_n}}
	$$
	which by definition satisfies that $L_n\le \frac{S_{2^{k_n}}}{2^{2k_n}}$ and that
	$$
	\frac{S_{2^{k_n}}}{2^{2k_n}}=L_n+\frac{S_{2^{k_{n-1}}}}{2^{2k_n}}.
	$$
	Now recalling that under event $E$, $S_{2^n}< 2^{(3+\delta)n}$ for all large enough $n$, we have 
	\bae
	\label{eq_previous_term}
	\frac{S_{2^{k_{n-1}}}}{2^{2k_n}}\le \frac{2^{(3+\delta)k_{n-1}}}{2^{4k_{n-1}}}\le 1.
	\eae
	And for each $L_n$, 
	$$
	\prob(L_n\le M)\ge \prob\left(\frac{S_{2^{k_n}}}{2^{2k_n}}\le M \right)\ge p_0
	$$
	and by independence of $L_n$ one immediately have that $L_n\le M$ a.s. has a positive density, and certainly infinitely often. This together with \eqref{eq_previous_term} show that we with probability one have $\frac{S_{2^{k_n}}}{2^{2k_n}}\le M+1$ happens i.o., which concludes the proof. 
\end{proof}

In fact, the proof above has also given an iterated exponential upper bound on the growth rate of $\CE_t$ at criticality: 

\begin{corollary}
	\label{col_iterated_exp}
	There is a constant $C<\infty$ such that asymptotically a.s. we have $\CE_t\le 2^{\exp(Ct)}$ as $t\to \infty$. 
\end{corollary}

\section{Sub-critical growth rate (Theorem \ref{thm:growth} (1)) }

\begin{proof}[{Proof of the lower bound in Theorem \ref{thm:growth} (1)}]

\

To prove the lower bound, we first note that $\CE_t$ is monotonically non-decreasing. So it suffices to show that for all $\epsilon>0$
\beq
\label{eq_lower_dis}
\prob\left(\CE_k\le k^{\frac{1}{2}}\log^{-\frac{1+\epsilon}{2}}k, \ i.o.\right)=0. 
\eeq
Define r.v. $\hat N_k$ to be the number of particles that have ever entered $B(r_0)$ by time $k$. Then by definition of the engulfing model, one always has $\hat N_k\le \pi\CE_k^2$. Thus it suffices to show that 
\beq
\label{eq_num_enter}
\sum_{k=1}^\infty \prob\left(\hat N_k\le \pi k\log^{-1-\epsilon}k \right)<\infty.
\eeq
Now note that $\hat N_k$ forms a Poisson r.v. with intensity 
$$
\lambda_k=\lambda \left(1+\iint_{B(r_0)^c}\prob_{(x,y)}\left(\tau_{B(r_0)<k} \right) dxdy \right).
$$
Thus the proposed lower bound is immediate once we prove the following lemma:

\begin{lemma}
	\label{lem_integrand}
	For all $\epsilon>0$, there is some $c>0$ and $k_0<\infty$ such that for all $k\ge k_0$ and $(x,y)\in B\left(k^{\frac{1}{2}}\log^{-\frac{\epsilon}{3}}k\right)$, there is
	$$
	\prob_{(x,y)}\left(\tau_{B(r_0)}<k\right)\ge \frac{c}{\log k}.
	$$
\end{lemma}
\begin{proof}
	Note that the random walk $X_t$ may typically move a distance of $k^{\frac{1}{2}}$ over time $k$. Define stopping time
	$$
	\bar \tau_k=\tau_{\partial B(k^{\frac{1}{2}}\log^{-\frac{\epsilon}{4}}k)}.
	$$
	Then
	$$
		\prob_{(x,y)}\left(\tau_{B(r_0)}<k\right)\ge \prob_{(x,y)}\left(\tau_{B(r_0)}<\bar \tau_k\right)-\prob_{(x,y)}\left(\bar \tau_k>k \right).
	$$
	For the second term, note that there exists an absolute constant $p_0>0$ such that for all $(x,y)\in B(k^{\frac{1}{2}}\log^{-\frac{\epsilon}{4}}k)$
	$$
	\prob_{(x,y)}\left(\bar\tau_k\le k\log^{-\frac{\epsilon}{2}}k  \right)\ge p_0.
	$$
	Thus by strong Markov property, for any $(x,y)\in B\left(k^{\frac{1}{2}}\log^{-\frac{\epsilon}{3}}k\right)$, 
	$$
	\prob_{(x,y)}\left(\bar \tau_k>k \right)\le (1-p_0)^{\log^{\frac{\epsilon}{2}}k}=s.e.(\log k)\ll \log k. 
	$$
	Thus we only need to control the first term and prove that 
	\beq
	\label{eq_option}
	\prob_{(x,y)}\left(\tau_{B(r_0)}<\bar \tau_k\right)\ge \frac{c}{\log k}.
	\eeq
	To show \eqref{eq_option}, recalling the definition of continuous random walk $\br_t=\sum_{i=1}^{N(t)} v_i$, its embedding chain is identically distributed as a standard Brownian motion at integer time points. Thus
	$$
	\prob_{(x,y)}\left(\tau_{B(r_0)}<\bar \tau_k\right)\ge \prob_{(x,y)}\left(\hat \tau_{\frac{r_0}{2}}<\hat \tau_{k^{\frac{1}{2}}\log^{-\frac{\epsilon}{4}}k} \right)\max_{(x,y)\in B(\frac{r_0}{2})} \prob_{(x,y)}\left(\max_{t\le 1}|B_t|\le r_0 \right)
	,$$
	where $\hat \tau_r$ is the stopping time that a standard Brownian motion first hits $\partial B(R)$. 
	
	For fixed $r_0=\sqrt{1/\pi}$, it is immediate that the second term in the inequality above is a strictly positive constant. And for the first one, note that for 1 standard Brownian motion $B_t, \ t\ge 0$ in $\BR^2$, $\log|B_t|$ forms a martingale by stopping time $\hat \tau_{\frac{r_0}{2}}\wedge \hat \tau_{k^{\frac{1}{2}}\log^{-\frac{\epsilon}{4}}k}$. Thus by Doob's optional stopping theorem, 
	we have for all $(x,y)\in B\left(k^{\frac{1}{2}}\log^{-\frac{\epsilon}{3}}k\right)$
	\begin{align*}
		\log\left(k^{\frac{1}{2}}\log^{-\frac{\epsilon}{3}}k \right)\ge \log\sqrt{x^2+y^2}&=\log\left(\frac{r_0}{2}\right)\prob_{(x,y)}\left(\hat \tau_{\frac{r_0}{2}}<\hat \tau_{k^{\frac{1}{2}}\log^{-\frac{\epsilon}{4}}k} \right)\\
		&+\log\left(k^{\frac{1}{2}}\log^{-\frac{\epsilon}{4}}k \right)\left[1-\prob_{(x,y)}\left(\hat \tau_{\frac{r_0}{2}}<\hat \tau_{k^{\frac{1}{2}}\log^{-\frac{\epsilon}{4}}k} \right) \right].
	\end{align*}
	Thus we have 
	$$
	\prob_{(x,y)}\left(\hat \tau_{\frac{r_0}{2}}<\hat \tau_{k^{\frac{1}{2}}\log^{-\frac{\epsilon}{4}}k} \right)\ge \frac{c\log\log k}{\log k}\ge \frac{1}{\log k}
	$$
	for all sufficiently large $k$, which concludes the proof of this lemma.
\end{proof}

Now back to the proof of the sub-critical lower bound, recall that 
$$
\lambda_k=\lambda \left(1+\iint_{B(r_0)^c}\prob_{(x,y)}\left(\tau_{B(r_0)<k} \right) dxdy \right).
$$
By Lemma \ref{lem_integrand}, we have the integrand is uniformly bounded from below by $\frac{c}{\log k}$ in $B\left(k^{\frac{1}{2}}\log^{-\frac{\epsilon}{3}}k\right)\setminus B(r_0)$. Thus
$$
\lambda_k\ge \pi k\log^{-\frac{2\epsilon}{3}}k \times \frac{c}{\log k}\gg k\log^{-(1+\epsilon)}k.
$$ 
Thus by large deviation estimate of Poisson r.v., one immediately has \eqref{eq_num_enter} and thus conclude of the proof of the lower bound in Theorem \ref{thm:growth} (1). 

\end{proof}

In fact, we have also shown the following quantitative estimate: 

\begin{corollary}
\label{col_low}
For any $\epsilon>0$, one always has 
$$
\prob\left(\hat N_R\le R^{-1-\epsilon} \right)\le \prob\left(\hat N_R\le \pi R\log^{-1-\epsilon}R \right)=s.e.(R)
$$
for all $R$ large enough. 
\end{corollary}

\begin{proof}[{Proof of the upper bound in Theorem \ref{thm:growth} (1)}]
	Recall that $A_0$ forms a Poisson point process on $\BR^2$ with constant intensity $\lambda$. For any Borel set $A\subset \BR^2$ with finite volume, the expected number of particles in $A_0\cap A$ is thus a Poisson r.v.'s with $\lambda(A)=\lambda \text{Vol}(A)$. It is not hard to see the follow large deviation estimate on $A_0$ intersecting balls:
	
	\begin{lemma}
		\label{lem_large_dev_up}
		For all fixed $\lambda>0$ and any $\epsilon>0$, there is an event $E_{\epsilon}$ such that $\prob(E_{\epsilon})=1$ and that for all $\omega\in E_{\epsilon}$ there is some $R_0(\omega)<\infty$ so that 
		\beq
		\label{eq_large_dev_up}
		|A_0(\omega)\cap B(R)|<(\lambda+\epsilon) \pi R^2, \ \forall R\ge R_0(\omega).
		\eeq
	\end{lemma}
	
	\begin{proof}
		Define events
		$$
		G_{\epsilon,n}=\left\{|A_0\cap B(n)|\ge\left(\lambda+\frac{\epsilon}{2}\right) \pi nR^2\right\}, \ \forall n\in \BZ^+,
		$$
		and $E_{\epsilon}=\{G_{\epsilon,n}, \ i.o.\}^c$. Then by monotonicity, it is straightforward to see that $\forall \omega\in E_{\epsilon}$, \eqref{eq_large_dev_up} is satisfied. Thus by Borel-Cantelli, it suffices to show that $\prob(G_{\epsilon,n})$ is summable. On the other hand, recalling that 
		$$
		|A_0(\omega)\cap B(n)|\sim \Pois\left(\lambda  \pi n^2 \right),
		$$
		by standard large deviation estimate (see Theorem 2.7.7., \cite{durret2019probability} for instance), there is some $c=c(\lambda,\epsilon)>0$ so that $\prob(G_{\epsilon,n})\le \exp(-c n^2)$, which is clearly summable over $n$. 
	\end{proof}
	
	%%%%%%%%%%%%%%%%%%%%%%%%%%%%%%%%%%%%%%%%%%%%%%%%%%%%%%%%%%%%%%%%%%%%%%%%%%%%%%%%%%%%%%%%%%%%%%%%%%%%%%%%%%%%%%%%%%%%%%%%%%%%%%%%%%%%%%%%%%%%%%%%%%%%%%%%%%%%%%%%%%%%%%%%%%%%%%%%%%%%%%%%%
	
	Now back to the proof of the upper bound, we first note that it follows the argument of Kesten and Sidravucius \cite{kesten2008problem}. Denote $V(r,t)$ to be the number of particles that move into $B(r)$ in the time interval $[0,t]$, i.e.
	\bae
	V(r,t)=|\{x\in A_0\setminus B(r):\inf_{s\le t}|\br^x_s|<r\}|.
	\eae
	By possibly over-counting particles that initiate in $B(\CE_t)$ but do not actually contribute to the growth, we obtain
	\bae
	\pi\CE_t^2\le |A_0\cap B(\CE_t)|+V(\CE_t,t).
	\eae
	Recalling Lemma \ref{lem_large_dev_up}, and that $\CE_t\to\infty$ as $t\to\infty$ by the lower bound we just showed, for each fixed $\epsilon>0$, there is some random $t_0$ such that for all $t>t_0$, $|A_0\cap B(\CE_t)|\le (\lambda+\epsilon)\pi\CE_t^2$.
	We thus get that $\forall t>t_0$, 
	\bae\label{eq:Vlower}
	V(\CE_t,t)\ge(1-\lambda-\epsilon)\pi \CE_t^2
	.\eae
	Next we estimate the expected number of particles that enter a ball. Denote $\br^{(1)}$, one dimensional Brownian motion. 
	\bae
	\ev[V(B(R),t)]&=\int_{r=R}^{\infty}\lambda 2\pi r \prob_r\left[\inf_{s\le t}|\br(s)|<R\right]dr\\
	&\le\int_{r=R}^{\infty}\lambda 2\pi r \prob_r\left[\inf_{s\le t}|\br^{(1)}(s)|<R\right]dr\\
	&\le C\lambda (R\sqrt{t}+t)
	.\eae
	Thus by Markov's inequality,
	\bae
	\prob\left[V(\CE_t,t)\ge (1-\lambda-\epsilon)\pi\CE_t^2|\CE_t\right]\le\frac{\ev\left[V(\CE_t,t)|\CE_t\right]}{(1-\lambda-\epsilon)\pi\CE_t^2}\le \frac{ C\lambda (\CE_t\sqrt{t}+t)}{(1-\lambda-\epsilon)\pi\CE_t^2}
	.\eae
	consider the sequence $t_k=4^k$, $R_k= k^2 2^k$. We obtain that
	\bae
	\prob\left[V(\CE_{t_k},{t_k})\ge (1-\lambda-\epsilon)\pi\CE_{t_k}^2|\CE_{t_k}\ge R_k\right]\le \frac{c}{k^2}
	.\eae
	Thus by Borel-Cantelli $\{V(\CE_{t_k},{t_k})\ge (1-\lambda-\epsilon)\pi\CE_{t_k}^2\}\cap\{\CE_{t_k}\ge R_k\}$ occurs for finitely many $k$'s a.s. By \eqref{eq:Vlower} we obtain that a.s. $\forall t>t_0$, $\CE_t\le 2\sqrt{t}(\log t)^2$.
\end{proof}

%%%%%%%%%%%%%%%%%%%%%%%%%%%%%%%%%%%%%%%%%%%%%%%%%%%%%%%%%%%%%%%%%%%%%%%%%%%%%%%%%%%%%

\section{Critical growth rate (Theorem \ref{thm:growth} (2)) }

In this section, we prove the second part of Theorem \ref{thm:growth} according to the following schemes: In Subsection \ref{subsec_volume}, we first show that, for our homogeneous Poisson cloud $A_t$, it is asymptotically impossible that the number of particles within an expanding ball deviates significantly from its expectation. However, Proposition \ref{prop_slow} shows that under the assumption of a sublinear power growth, there will be infinitely many occasions of ``slow" growth, which by Theorem \ref{thm:kurtz} implies the existence of abnormally small number of particles, and thus lead to a contradiction. 

\subsection{Volume Estimates of Poisson Field}
\label{subsec_volume}

To depict the ``ordinary" behavior of $A_t$, we start with the following lemma which shows that it is very unlikely that too many particles can ever escape from a given large ball. To be precise, for all $n$, define
$$
Y_n=\left|\left\{i: X^i_0\in B(n), \ \max_{s\le 1}|X^i_s|>n\right\} \right|
$$

\begin{lemma}
	\label{lem_exit}
	For all $\delta>0$, there is some $\gamma>0$ such that 
	$$
	\prob(Y_n\ge n^{1+\delta})\le \exp(-n^\gamma)
	$$
	holds for all sufficiently large $n$. 
\end{lemma}

\begin{proof}
	Denote by $E_n$ the event of interest. Then $E_n\subset E_{n,1}\cup E_{n,2}$ where
	$$
	E_{n,1}=\{|A_0\cap B(n)\setminus B(n-n^{\frac{\delta}{2}})|\ge n^{1+\delta}\}
	$$
	and
	$$
	E_{n,2}=\{\exists i \ s.t. X^i_0\in B(n-n^{\frac{\delta}{2}}) \text{ and } \max_{s\le 1}|X^i_s|>n \}.
	$$
	For $E_{n,1}$, note that $|A_0\cap B(n)\setminus B(n-n^{\frac{\delta}{2}})|$ follows a Poisson distribution of which the intensity equals to 
	$$
	\lambda \text{Vol}[ B(n)\setminus B(n-n^{\frac{\delta}{2}})]= \pi\lambda[n^2-(n-n^{\frac{\delta}{2}})^2]\le 4\pi\lambda n^{1+\frac{\delta}{2}}.
	$$
	Then by a standard large deviation estimate of Poisson random variables, there is an absolute constant $c>0$ such that $\prob(E_{n,1})\le \exp(-c  n^{1+\frac{\delta}{2}})$ for all integer $n\ge 1$. And for $E_{n,2}$, a simple union bound implies that 
	\begin{align*}
	\prob(E_{n,2})&\le \prob(|A_0\cap B(n-n^{\frac{\delta}{2}})|\ge 4\pi n^2)+4\pi n^2\max_{x\in B(n-n^{\frac{\delta}{2}})} \prob_x \left(\max_{s\le 1}|\br_s|\ge n\right)\\
	& \le \prob(|A_0\cap B(n)|\ge 4\pi n^2)+4\pi n^2 \prob_0 \left(\max_{s\le 1}|\br_s|\ge n^{\frac{\delta}{2}}\right).
	\end{align*}
	And the desired result follows directly again from standard decay estimate Poisson and norm distributions. 
\end{proof}

Now for $\delta,R>0$, define $A_{R,\delta}$ as the event that there is ever an occasion with an abnormally small number of particles in $B(R)$ in the time interval $[0, R^{10}]$. I.e., for some $\delta>0$, let
$$
A_{R,\delta}=\left\{\exists t\in [0,R^{10}], s.t. \ |A_t\cap B(R)|<\pi R^2-R^{1+\delta} \right\}
.$$

The following proposition shows that such events will asymptotically a.s. stop happening as $R\to\infty$. 

\begin{proposition}
	\label{prop_abnormal}
	For all $\delta>0$. 
	$$
	\prob\left(\bigcup_{k=1}^\infty \bigcap_{R\ge k}  A_{R,\delta}^c \right)=1.
	$$
	I.e., almost surely there is some $R_0(\omega)<\infty$ such that for all $R\ge R_0$, $ |A_t\cap B(R)|\ge\pi R^2-R^{1+\delta}, \ \forall t\in [0,R^{10}]$. 
\end{proposition}

\begin{proof}
	For all integer $n$, one may similarly define
	$$
	\bar A_{n,\delta}=\left\{\exists t\in [0,2n^{10}], s.t. \ |A_t\cap B(n)|<\pi n^2-n^{1+\delta}  \right\}.
	$$	
	By definition, for all sufficiently large $R$ such that $\lfloor R\rfloor\ge \frac{R}{2^{\frac{1}{10}}}$, we always have $A_{R,\delta}\subset \bar A_{n,\delta}$. Thus to prove this proposition, it suffices to show the following assertion: 
	\beq
	\label{eq_bar_A}
		\bar A_{n,\delta}=s.e.(n)
	\eeq
	and thus summable w.r.t. $n$ for all $\delta>0$. Now for each $n$ it is not hard to see that 
	$$
	\bar A_{n,\delta}=\bigcup_{k=0}^{2n^{10}-1} \tilde A^{(k)}_{n,\delta}\cup \bigcup_{k=1}^{2n^{10}}\hat A^{(k)}_{n,\delta}
	$$
	where 
	$$
	\tilde  A^{(k)}_{n,\delta}=\left\{|A_k\cap B(n)|<\pi n^2-n^{1+\frac{\delta}{4}} \right\}
	$$
	stands for the event there are abnormally small number of particles at a given integer time point, and 
	\begin{align*}
	\hat  A^{(k)}_{n,\delta}=&\{\text{there are at least $n^{1+\frac{\delta}{4}}$ particles that have ever moved from} \\
	& \text{$B(n)$ to $B(n)^c$ in time interval $[k-1,k]$} \}
	\end{align*}
	for the event that too many particles escape the ball in a unit time interval. 
	
	By Lemma \ref{lem_exit}, we showed that $\prob(\hat A^{(k)}_{n,\delta})=s.e.(n)$. And for $\tilde  A^{(k)}_{n,\delta}$ note that
	$$
	Z_k:=\left|A_k\cap B(n) \right|\sim \Pois(\pi n^2).
	$$ 
	Thus by the concentration inequality of Poisson distribution, 
	$$
	\prob\left(Z_k\le \pi n^2-n^{1+\frac{\delta}{4}}\right)\le \exp\left(-\frac{n^{2+\frac{\delta}{2}}}{2\left(\pi n^2+n^{1+\frac{\delta}{4}}\right)} \right)\le \exp\left(-\frac{n^{\frac{\delta}{2}}}{4\pi} \right)
	$$
	which concludes the proof of this proposition. 
\end{proof}

\subsection{Mesoscopic Stall under Sublinear Growth}

In this section, we show that if $\CE_t$ grows between two sub-linear power orders, then there will be infinite many occasions where $\CE_t$ grows no more than $O(1)$ for a mesoscopic time scale. 

\begin{proposition}
	\label{prop_slow}
	For an arbitrary C\'adl\'ag function $\CE_t$ such that 
	\beq
	\label{assump_low}
	\liminf_{t\to\infty} \frac{\CE_t}{t^{\frac{1}{3}}}=\infty
	,\eeq
	and that 
	\beq
	\label{assump_up}
	\limsup_{t\to\infty} \frac{\CE_t}{t^{1-\alpha}}\le 1
	\eeq
	for some $\alpha\in (0,1)$, there is some $\beta>0$ such that for all $T_0\ge 0$ there is a $t_1>T_0$ such that 
	\beq
	\label{eq_slow}
	\CE_{t_1-t_1^\beta}\ge \CE_{t_1}-2.
	\eeq
\end{proposition}

\begin{remark}
	For the engulfing pool model with $\lambda=1$, the lower bound in Theorem \ref{thm:growth} (1) already gives the desired $\liminf$ in this proposition. And if we assume there is also a sublinear power order of upper bound, it will satisfy the $\limsup$. 
\end{remark}

\begin{proof}
	Without loss of generality, it suffices to prove this proposition for all $T_0$ larger enough. By \eqref{assump_low}, we hereby assume $T_0\ge 100$ and that $\CE_{T_0}\ge T_0^{\frac{1}{3}}$. Define
	$$
	\varphi(t)=(t-T_0)^{1-\frac{\alpha}{2}}.
	$$
	By \eqref{assump_up}, $\lim_{t\to\infty} \frac{\CE_t}{\varphi(t)}=0$, thus
	$$
	t_1=\inf\{t\ge T_0, \ \varphi(t)\ge \CE_t\}<\infty.
	$$
	For $\Delta t_1=t_1-T_0$, one may note that 
	$$
	\varphi(t_1)=(\Delta t_1)^{1-\frac{\alpha}{2}}\ge\CE_{t_1}\ge \CE_{T_0}\ge T_0^{\frac{1}{3}},
	$$
	which implies that 
	$$
	\Delta t_1\ge T_0^{\frac{1}{3(1-\frac{\alpha}{2})}}>1
	$$
	and that $T_0\le \Delta t_1^{3(1-\frac{\alpha}{2})}$. So we have 
	$$
	t_1=T_0+\Delta t_1\le \Delta t_1+\Delta t_1^{3(1-\frac{\alpha}{2})}\le 2\Delta t_1^{3(1-\frac{\alpha}{2})}
	$$
	where the second inequality holds as $3(1-\frac{\alpha}{2})\ge \frac{3}{2}>1$. Thus we have 
	\beq
	\label{eq_low_delta}
	\Delta t_1\ge \left(\frac{t_1}{2} \right)^{\frac{1}{3(1-\frac{\alpha}{2})}}
	\eeq
	which implies that $\Delta t_1$ is of at least a mesoscopic power order w.r.t. $t_1$. 
	
	Now recalling the definition of $t_1$ we have 
	$$
	\CE_{t_1-s}\ge \varphi(t_1-s)=(\Delta t_1-s)^{1-\frac{\alpha}{2}}
	$$
	for all $s\in [0,\Delta t_1]$. This also indicates that 
	$$
		\CE_{t_1}-\CE_{t_1-s}\le (\Delta t_1)^{1-\frac{\alpha}{2}}-(\Delta t_1-s)^{1-\frac{\alpha}{2}}.
	$$ 
	So now let $\gamma=\frac{\alpha}{4}$, by mid value theorem
	$$
	\begin{aligned}
		\CE_{t_1}-\CE_{t_1-\Delta t_1^\gamma}&\le (\Delta t_1)^{1-\frac{\alpha}{2}}-(\Delta t_1-\Delta t_1^\gamma)^{1-\frac{\alpha}{2}}\\
		&\le \Delta t_1^\gamma \left(1-\frac{\alpha}{2}\right) \left(\frac{\Delta t}{2} \right)^{-\frac{\alpha}{2}}\le 2 \left(\Delta t \right)^{-\frac{\alpha}{4}}<2.
	\end{aligned}
	$$
	Finally, recalling \eqref{eq_low_delta}, we have $\Delta t_1^\gamma\ge t_1^\beta$ with $\beta=\frac{1}{12(1-\frac{\alpha}{2})}$, which concludes the proof of this proposition. 
\end{proof}

\begin{proof}[Proof of Theorem \ref{thm:growth} (2) ]
	Now we are ready to prove our main result on the lower bound of growth at criticality. For the sake of contradiction, for a $\delta>0$ consider the event 
	$$
	\tilde E_1=\left\{\limsup_{t\to\infty} \frac{\CE_t}{t^{1-\delta}}<1\right\}
	,$$
	and assume that $\prob(\tilde E_1)>0$. By \ref{thm:growth}(1), one also has for each $\gamma\in (\frac{1}{3}, \frac{1}{2})$, $\prob(\tilde E_2)=1$, where 
	$$
	\tilde E_2=\left\{\liminf_{t\to\infty} \frac{\CE_t}{t^{\gamma}}=+\infty\right\}.
	$$
	So, under event $\tilde E_1\cap \tilde E_2$, Conditions (\ref{assump_low}, \ref{assump_up}) are satisfied. Thus we can define the following sequence of stopping times w.r.t. to $\sigma(\{\CE_s, \ s\le t\})$ where 
	$$
		\tilde \tau_1=\inf\left\{t\ge 1, \ \CE_{t-t^\beta}\ge \CE_t-2 \right\}.
	$$
	And for all $k\ge 1$, we iteratively define 
	$$
		\tilde\tau_{k+1}=\inf\left\{t\ge \tilde\tau_{k}+1, \ \CE_{t-t^\beta}\ge \CE_t-2  \right\}. 
	$$
	Thus $\tilde\tau_k\ge k, \ \forall k\ge 1$. By Proposition \ref{prop_slow}, in $\tilde E_1\cap \tilde E_2$, we have that $\tilde\tau_n<\infty$ for all $n\ge 1$. Moreover, for each $n\ge 1$, denote by $\tilde R_n=\CE_{\tilde\tau_n}$ the radius of our engulfing model at $\tilde\tau_n$, with the convention that $\tilde R_n=\infty$ if $\tilde \tau_n=\infty$. 
	
	Then for all $\omega\in \tilde E_1\cap \tilde E_2$, by definition there must exist some finite integer $K_0(\omega)$ such that $\CE_t\in (t^\gamma, t^{1-\delta})$ for all $t\ge K_0$, implying that 
	$$
	\tilde R_k \in \left(\tilde \tau_k^\gamma,\tilde \tau_k^{1-\delta} \right), \ \forall k\ge K_0.
	$$
	Thus we have 
	$$
	\tilde E_1\cap \tilde E_2\subset \bigcup_{n=1}^\infty \bigcap_{k=n}^\infty \left\{ \tilde R_k\in  \left(\tilde \tau_k^\gamma,\tilde \tau_k^{1-\delta} \right),\tilde \tau_k<\infty\right\}
	.$$
	So, under the assumption that $\prob(\tilde E_1)>0$, there must exist some $N_0<\infty$ such that 
	\beq
	\label{eq_Rk}
	\prob\left(\bigcap_{k=N_0}^\infty \left\{ \tilde R_k\in  \left(\tilde \tau_k^\gamma,\tilde \tau_k^{1-\delta} \right),\tilde \tau_k<\infty\right\} \right)>0.
	\eeq
	Denote such probability by $c_0$. Now define $\tilde N_k=|A_{\tilde \tau_k}\cap B(2\tilde R_k)|$ to be the number of particles, regardless of their labels, within $B(2\tilde R_k)$ at time $\tilde \tau_k$, and 
	$$
	\tilde M_k=\left|\{i: a(i,\tilde \tau_k)=1, \br^i_{\tilde \tau_k}\in B(2\tilde R_k)\} \right|
	$$
	to be the number of free particles in $ B(2\tilde R_k)\setminus B(\tilde R_k)$. Note that by conservation of mass, $\pi \tilde R_k^2$ equals to the total number of particles that have released their mass. Thus
	\beq
	\label{eq_vol_1}
	\tilde N_k\le \pi\tilde R_k^2+ \tilde M_k.
	\eeq
	Now consider events:
	$$
	H_k=\left\{ \tilde R_k\in  \left(\tilde \tau_k^\gamma,\tilde \tau_k^{1-\delta} \right),\tilde \tau_k<\infty \right\}\cap \left\{\tilde M_k\ge 3\pi \tilde R_k^2-\tilde R_k^{1+\frac{\beta}{4}} \right\}
	$$
	where $\beta$ is the constant in Proposition \ref{prop_slow}. We hereby claim that 
	\beq
	\label{eq_H}
	\prob(H_k)=s.e.(k),
	\eeq
	and thus summable over $k$. To prove this claim, noting that 
	$$
	\left\{ \tilde R_k\in  \left(\tilde \tau_k^\gamma,\tilde \tau_k^{1-\delta} \right),\tilde \tau_k<\infty \right\}\in \sigma(\tilde R_k) \subset  \sigma(\{\CE_t, t\le \tilde \tau_k\})
	,$$
	by Theorem \ref{thm:kurtz}, given the growth history of the radius, the (conditional) distribution of free particles in $B(\tilde R_k)^c$ form a Point point process with intensity 
	\beq
	\label{eq_con_int}
	\tilde \lambda^{(k)}\left(x,y| \CE_{t}, \ \forall t\le \tilde \tau_k\right)=\prob_{(x,y)} \left(|\br_t|\ge \CE_{\tilde\tau_k-t}, \ \forall t\le \tilde \tau_k \right). 
	\eeq
	Thus we have the conditional distribution $\tilde M_k$ is given by $\Pois\left(\tilde \Lambda^{(k)}\left(\CE_{t}, \ \forall t\le \tilde \tau_k\right)\right)$, where
	$$
	\tilde \Lambda^{(k)}\left(\CE_{t}, \ \forall t\le \tilde \tau_k\right)= \int_{B(2\tilde R_k)\setminus B(\tilde R_k)} 	\tilde \lambda^{(k)}\left(x,y| \CE_{t}, \ \forall t\le \tilde \tau_k\right) dxdy.
	$$
	Recall \eqref{prop_slow}, and that under event $\left\{ \tilde R_k\in  \left(\tilde \tau_k^\gamma,\tilde \tau_k^{1-\delta} \right),\tilde \tau_k<\infty \right\}$, $\tilde R_k<\tilde \tau_k^{1-\delta}\le \tilde \tau_k$. We have for all 
	$$
	(x,y)\in B\left(\tilde R_k+\tilde R_k^{\frac{\beta}{2}} \right)\setminus B\left(\tilde R_k\right)\subset B\left(\tilde R_k+\tilde \tau_k^{\frac{\beta}{2}} \right)\setminus B\left(\tilde R_k\right),
	$$
	by invariance principle, there is some absolute constant $c_0\in (0,1)$ such that 
	\beq
	\label{eq_inv_1}
	\begin{aligned}
		\prob_{(x,y)} \left(|\br_t|\ge \CE_{\tilde\tau_k-t}, \ \forall t\le \tilde \tau_k \right)&\le \prob_{(x,y)} \left(|\br_t|\ge \CE_{\tilde\tau_k-\tilde\tau_k^\beta}, \ \forall t\le \tilde \tau_k^\beta \right)\\
	&\le \prob_{(x,y)} \left(|\br_t|\ge \CE_{\tilde\tau_k}-2, \ \forall t\le \tilde \tau_k^\beta \right)\le 1-c_0.
	\end{aligned}
	\eeq
	Thus we have 
	$$
	\int_{B\left(\tilde R_k+\tilde R_k^{\frac{\beta}{2}} \right)\setminus B\left(\tilde R_k\right)} 	\tilde \lambda^{(k)}\left(x,y| \CE_{t}, \ \forall t\le \tilde \tau_k\right) dxdy\le (1-c_0) V\left(B\left(\tilde R_k+\tilde R_k^{\frac{\beta}{2}} \right)\setminus B\left(\tilde R_k\right) \right)
	$$
	which implies that 
	\begin{align*}
	\tilde \Lambda^{(k)}\left(\CE_{t}, \ \forall t\le \tilde \tau_k\right)&\le 3\pi \tilde R_k^2-c_0 V\left(B\left(\tilde R_k+\tilde R_k^{\frac{\beta}{2}} \right)\setminus B\left(\tilde R_k\right) \right)\\
	&\le  3\pi \tilde R_k^2-2c_0\pi \tilde R_k^{1+\frac{\beta}{2}}.
	\end{align*}
	Thus again by concentration inequality of Poisson r.v., there is some absolute constant $C<\infty$,  
	\beq
	\label{eq_con_2}
		\prob\left(\left.\tilde M_k\ge 3\pi \tilde R_k^2-\tilde R_k^{1+\frac{\beta}{4}}\right|\CE_{t}, \ \forall t\le \tilde \tau_k \right)\le \exp \left(-\frac{\tilde R_k^{2+\beta}}{C\tilde R_k^2} \right)=s.e.(\tilde R_k)
	\eeq
	under event $\left\{ \tilde R_k\in  \left(\tilde \tau_k^\gamma,\tilde \tau_k^{1-\delta} \right),\tilde \tau_k<\infty \right\}$. Now by total probability formula and recall that $\tilde \tau_k\ge k$, one has 
	$$
	\prob(H_k)=E\left[\ind_{\left\{ \tilde R_k\in  \left(\tilde \tau_k^\gamma,\tilde \tau_k^{1-\delta} \right),\tilde \tau_k<\infty \right\}}\prob\left(\left.\tilde M_k\ge 3\pi \tilde R_k^2-\tilde R_k^{1+\frac{\beta}{4}}\right|\CE_{t}, \ \forall t\le \tilde \tau_k \right)\  \right]=s.e.(k)
	$$
	and thus verifies the assertion in \eqref{eq_H}. And combining with \eqref{eq_Rk}, there is some $N_1\ge N_0$ such that for
	\begin{align*}
		J_{N_1}&:= \left(\bigcap_{k=N_1}^\infty \left\{ \tilde R_k\in  \left(\tilde \tau_k^\gamma,\tilde \tau_k^{1-\delta} \right),\tilde \tau_k<\infty\right\} \right)\setminus \left(\bigcup_{k=N_1}^\infty H_k \right)\\
		&=\bigcap_{k=N_1}^\infty \left(\left\{ \tilde R_k\in  \left(\tilde \tau_k^\gamma,\tilde \tau_k^{1-\delta} \right),\tilde \tau_k<\infty\right\} \cap \left\{\tilde M_k< 3\pi \tilde R_k^2-\tilde R_k^{1+\frac{\beta}{4}} \right\}  \right)
	\end{align*}
	we have $\prob(J_{N_1})>0$. 
	
	However, recalling \eqref{eq_vol_1}, under event $J_{N_1}$ for all $k\ge N_1$
	$$
		\tilde N_k\le \pi\tilde R_k^2+ \tilde M_k\le 4\pi \tilde R_k^2-\tilde R_k^{1+\frac{\beta}{4}}.
	$$
	At the same time, under $J_{N_1}$ we have $\tilde R_k>\tilde \tau_k^\gamma, \ \forall k\ge N_1$, which implies that $\tilde \tau_k<\tilde R_k^3$. Thus we now have the increasing sequence $\tilde R_k\uparrow \infty$, such that for each $k$ there is some time $\tilde \tau_k \in [0,\tilde R_k^3)\subset [0,(2\tilde R_k)^{10})$ such that the number of particles in $B(2\tilde R_k)$ is less than $\pi(2 \tilde R_k)^2- \tilde R_k^{1+\frac{\beta}{4}}$ which contradicts with Proposition \ref{prop_abnormal} with $\delta=\frac{\beta}{5}$.  
\end{proof}

%%%%%%%%%%%%%%%%%%%%%%%%%%%%%%%%%%%%%%%%%%%%%%%%%%%%%%%%%%%%%%%%%%%%%%%%%%%%%%%%%%%%%%%%%%%%%%%%%%%%%%%%%%%%%%%%%%%%%%%%%%%%%%%%%%%%%%%%%%%%%%%%%%%%%%%%%%%%%%%%%%%%%%%%%%%%%%%%%%%%%%%%%%%%%%%%%%%%%%%%%%%%%%%%%%%%%%%%%%%%%%%%%%%%%%%%%%%%%%%%%%%%%%%%%%%%%%%%%%%%%%%%%%%%%%%%%%%%%%%%%%%%%%%%%%%%%%%%%%%%%%%%%%%%%%%%%%%%%%%%%%%%%%%%%%%%%%%%%%%%%%%%%%%%%%%%%%%%%%%%%%%

\section*{Acknowledgments}
YZ is supported by NSFC-12271010, and the Fundamental Research Funds for the Central Universities and the Research Funds of Renmin University of China 24XNKJ06.

\appendix

\section{Kurtz Theorem for Brownian motion}\label{sec:jurtz_brownian}
In this section we present for future reference a version of Kurtz's theorem for the case the Poisson process $A_t=\{\mathfs{D}^x_t\}_{x\in A_0}$, where $\mathfs{D}^x_t$ are distributed as i.i.d. Brownian motions.
\begin{theorem}\label{thm:kurtz_brownian}
For every $t>0$, the random collection of points $$\mathfrak{K}_t=\{x\in A_t: |x|>\CE_t, a(x,t)=1\},$$ conditional on $\sigma\{\CE_s:s\le t\}$ is an independent Poisson point process, with intensity measure $\lambda \prob_x(|\mathfs{D}_s|>\CE_{t-s}, \forall s\le t)$.% i.e. the dependence to the aggregate is only through the intensity measure.
\end{theorem}
%%%%%%%%%%%%%%%%%%%%%%%%%%%%%%%%%%%%%%%%%%%%%%%%%%%%%%%%%%%%%%%%%%%%%%%%%%%%%%%%%%%%%%%%%%%%%%%%%%%%%%%%%%%%%%%%%%%%%%%%%%%%%%
\begin{proof}
We partition $[0,t]$ to $2^k$ intervals. Let 
\bae
&\mathfrak{I}_1=\left\{x\in A_0: \mathfs{D}^x_{[0,2^{-k}t]}\cap B(1/\sqrt{\pi})\neq\emptyset\right\}
%&\mathfrak{I}_1=\left\{x\in A_0: \mathfs{D}^x_{[0,2^{-k}t]}\cap B\left(\sqrt{\frac{1+\mathfrak{I}'_1}{\pi}}\right)\neq\emptyset\right\}
.\eae
For any $1<j<2^k$, %\note{Need to include the engulfing procedure. Idea 27.2.25: all the processes for different $k$s are coupled via the same $A_t$. After the first engulfing there is a non-infinitesimal gap between each arrival. As $k\to\infty$ get a.s. convergence of the processes in the lower time gaps after each real engulfing procedure}
\bae\label{eq:discretetimeprocess_brownian}
\mathfrak{I}_j=\left\{x\in A_0\setminus \bigcup_{i\le j-1}\mathfrak{I}_i : \mathfs{D}^x_{[(j-1)2^{-k}t,j2^{-k}t]}\cap B\left(\sqrt{\frac{1+\sum_{i=1}^{j-1}|\mathfrak{I}_i|}{\pi}}\right)\neq\emptyset \right\}
.\eae
By the Marking Theorem \cite[Section 5.2]{kingman1992poisson}, $\left\{\mathfrak{I}_j\right\}_{j=1}^{2^k}$ are conditionally independent Poisson Point Processes: Conditional on $\{\mathfrak{I}_i\}_{i=1}^{j-1}$,  $\mathfrak{I}_j$ is an independent PPP with intensity,
\bae\label{eq:conditional density_brownian}
\lambda \prob_x&\left( \mathfs{D}^x_{[(j-1)2^{-k}t,j2^{-k}t]}\cap B\left(\sqrt{\frac{1+\sum_{i=1}^{j-1}|\mathfrak{I}_i|}{\pi}}\right)\neq\emptyset , \right.\\
&\left.  \forall l<j, ~\mathfs{D}^x_{[(l-1)2^{-k}t,l2^{-k}t]}\cap B\left(\sqrt{\frac{1+\sum_{i=1}^{l-1}|\mathfrak{I}_i|}{\pi}}\right)=\emptyset  \right) 
.\eae
%For the case of the annihilating process $\CA_t$ \note{not annihilating}, since the probability that any Brownian motion hits the aggregate in the interval $[t-2^{-k},t]$ is summable as $k\to\infty$. By Borel-Cantelli, the process $\mathfrak{I}_{2^k}$ converges a.s, and in fact, there is a finite a.s. $K>0$ such that for all $k_1,k_2>K$, $\mathfrak{I}_{2^{k_1}}=\mathfrak{I}_{2^{k_2}}=\mathfrak{K}_t$. By construction, we obtain that the conditional distribution of $\mathfrak{K}_t$ is as stated. 

We need to take into account the engulfing procedure. The idea is to show that we can encode the instantaneous engulfing procedure in the first several iterations of the discrete time process \eqref{eq:discretetimeprocess_brownian}. Recall the notations of Section \ref{Sec:intro}, and in particular \eqref{eq:engulfing_times}, of the successive particle arrival times that induce the engulfing procedure $\hat\tau_k$. 

Assuming that $\CE_{0}<\infty$, $N<\infty$ (for $t=0$ engulfing procedure), and $\hat\tau_1>0$ is well defined. Take $k_0$, large enough such that
\begin{equation}\label{eq:k0cond_brownian}
   N 2^{-k_0}<\hat\tau_1. 
\end{equation} 
First it is immediate that $A_0\cap B(r_0)\subset\mathfrak{I}_1$.
Next, by path-continuity of Brownian motion, for every $x\in A_0\cap B(r_1)\setminus B(r_0)$, we can find a $k_1>k_0$ large enough such that $\forall k>k_1$, $\left|\mathfs{D}^x_{[2^{-k},2\cdot2^{-k}]}\right|<r_1 $. Thus, for any $k>k_1$, $A_0\cap B(r_1)\setminus B(r_0)\subset \mathfrak{I}_1\cup \mathfrak{I}_2$. Continuing recursively, we get that, there exists a $k_N>k_{N-1}$, such that for any $k>k_N$, $A_0\cap B(r_N)\setminus B(r_{N-1})\subset \bigcup_{i=1}^N\mathfrak{I}_i$. On the other hand, by \eqref{eq:k0cond_brownian} and $k_N>k_0$, one obtains, by possibly choosing a $k$ larger such that particles that are not part of the engulfing procedure at $t=0$, will also not be in the first $N$, $\mathfrak{I}_i$'s,  $\sum_{i=1}^{N}|\mathfrak{I}_i|=\sum_{i=1}^N\xi_i$. Thus, if $t<\hat\tau_1$, then we are done. Otherwise, we obtain that for some $l_1(k)$, $\hat\tau_1\in[(l_1-1)2^{-k},l_12^{-k}]$, and assuming the engulfing procedure ends at a finite stage, for $k$ large enough $\hat\tau_2\notin[(l_1-1)2^{-k},(l_1+N)2^{-k}]$. By \eqref{eq:conditional density_brownian}, we obtain that conditional on $\{\mathfrak{I}_i\}_{i=1}^{l_1}$,  $\mathfrak{I}_{l_1+1}$ is an independent PPP with intensity,
\bae
\lambda \prob_x&\left( \mathfs{D}^x_{[l_1 2^{-k}t,(l_1+1)2^{-k}t]}\cap B\left(\sqrt{\frac{1+\sum_{i=1}^{l_1}|\mathfrak{I}_i|}{\pi}}\right)\neq\emptyset , \right.\\
&\left.  \forall l<l_1, ~\mathfs{D}^x_{[(l-1)2^{-k}t,l2^{-k}t]}\cap B\left(\sqrt{\frac{1+\sum_{i=1}^{l-1}|\mathfrak{I}_i|}{\pi}}\right)=\emptyset  \right) 
.\eae
By the strong Markov property of Brownian motion, we can now repeat the previous argument, by adjusting \eqref{eq:k0cond} to admit the number of engulfing steps in the time interval $\hat\tau_2-\hat\tau_1$. If all engulfing procedures end before time $t$, we are done. Otherwise we obtain that there are no active particles at time $t$, thus the statement follows trivially.  
\end{proof}
%%%%%%%%%%%%%%%%%%%%%%%%%%%%%%%%%%%%%%%%%%%%%%%%%%%%%%%%%%%%%%%%%%%%%%%%%%%%%%%%%%%%%%%%%%%%%%%%%%%%%%%%%%%%%%%%%%%%%%%%%%%%%%

We also include here the random stopping time version of Kurtz's theorem for the case of Brownian particles.
\begin{corollary}\label{cor:random_kurtz_brownian}
For every $k\in\BN$, the random collection of points $$\mathfrak{K}_{\hat\tau_k}=\{x\in A_{\hat\tau_k}: |x|>\CE_{\hat\tau_k}, a(x,{\hat\tau_k})=1\},$$ conditional on $\sigma\{\CE_s:s\le {\hat\tau_k}\}$ is an independent Poisson point process, with intensity measure $\lambda \prob_x(|\mathfs{D}_s|>\CE_{{\hat\tau_k}-s}, \forall s\le {\hat\tau_k})$.
\end{corollary}

%%%%%%%%
% \section{Useful claims about branching processes}\note{Do we need it?}
% The following is the exponential limit law of conditioned critical GW process \cite[Theorem 2]{athreya1972branching}.
% \begin{theorem}
% If $\lambda=1$ then,
% \bae
% \lim_{n\to\infty}\prob\left(\frac{Z_n}{n}>z\Big| Z_n>0\right)=e^{-cz},~ z\ge0.
% \eae
% \end{theorem}
% One conclusion is the the total progeny is of the same order as the extinction time. 

%\newpage
%\input{q argument}
%\newp
%\section{Setup}
%\input{setup}
%\input{upper boun
\bibliographystyle{alpha}
\bibliography{bibliography}

%\printbibliography
\end{document}